%% file: main.tex
\documentclass[9pt,journal]{IEEEtran}
\usepackage{cite} 
\usepackage{amsmath,amssymb,amsthm,graphics,epsfig,times,enumitem,float,mathtools,mathrsfs}
\usepackage{chngcntr}
\usepackage{apptools}
\AtAppendix{\counterwithin{lemma}{section}}
\usepackage{tikz}
\usetikzlibrary{arrows}
\usetikzlibrary{shapes}
\usetikzlibrary{arrows.meta} 
\usetikzlibrary{calc,intersections,through,backgrounds}
\usetikzlibrary{positioning}
\usepackage{pgfplots}
\usepgfplotslibrary{fillbetween}
\pgfplotsset{compat=1.16}
\usepackage{pgfplotstable}
\usepackage{filecontents}
\usepackage{multirow}
\usepackage{upgreek,bm}
\usepackage[linesnumbered,ruled]{algorithm2e}
\usepackage[caption=false]{subfig} 

\usepackage[dvipsnames]{xcolor}
\usepackage{hyperref}

\hypersetup{
    colorlinks=true,
    citecolor=blue,
    linkcolor=blue,
    urlcolor=blue
}
\theoremstyle{theorem}
\newtheorem{theorem}{Theorem} 
\newtheorem{corollary}{Corollary} 
\newtheorem{lemma}{Lemma} 
 
\theoremstyle{definition}
\newtheorem{definition}{Definition} 
\newtheorem{assumption}{Assumption} 
\newtheorem{remark}{Remark} 
 
\newcommand\normx[1]{\Vert#1\Vert}

\newcommand{\R}{\mathbb{R}}
\newcommand{\N}{\mathsf{N}}
\newcommand{\K}{\mathsf{K}}
\newcommand{\U}{\mathsf{U}}

\newcommand{\X}{\mathsf{X}}
\newcommand{\bu}{\mathsf{u}}

\newcommand{\col}[1]{\mathrm{col}\{#1\}}

\usepackage{textcomp}
\def\BibTeX{{\rm B\kern-.05em{\sc i\kern-.025em b}\kern-.08em
    T\kern-.1667em\lower.7ex\hbox{E}\kern-.125emX}}
\begin{document}

\title{Infinite-Horizon Linear-Quadratic Difference Games with Coupled Affine Inequality Constraints: Open-Loop Generalized Nash Equilibria}
\author{Partha Sarathi Mohapatra 
		\thanks{
		{Partha Sarathi Mohapatra is an independent researcher (e-mail: ps.mohapatra.work@gmail.com)}}} 	
	\date{ }
	\maketitle
    \thispagestyle{headings}
	\begin{abstract} 
        In this technical note, we study a class of infinite-horizon linear-quadratic difference games with coupled affine inequality constraints involving both state and control variables. We derive necessary conditions for the existence of open-loop generalized Nash equilibria and establish their sufficiency under additional assumptions. Sufficient conditions are characterized in terms of the existence of square-summable solutions to associated infinite-horizon linear complementarity systems. We further reformulate these conditions and show that computing open-loop generalized Nash equilibria reduces to solving a large-scale linear complementarity problem together with verifying additional conditions. Finally, we illustrate our results using a vehicle platooning example with constraints.
	\end{abstract}
\begin{IEEEkeywords}
      Difference games; infinite-horizon; coupled inequality constraints; open-loop generalized Nash equilibrium; linear complementarity problem 
\end{IEEEkeywords}
\section{Introduction}
Dynamic games provide a framework for modeling strategic interactions among multiple decision makers, where each player's actions influence both the system dynamics and the objectives of other players. The outcome of such interactions is commonly characterized by equilibrium concepts, most notably the Nash equilibrium, in which no player can improve its objective through unilateral deviation \cite{Basar:1999,Engwerda:2005}. Furthermore, in dynamic games, equilibrium strategies are shaped by the information available to players \cite{Basar:1999}, leading to open-loop formulations, in which strategies depend on time and the initial condition, and feedback formulations, in which they depend on the current state. Dynamic games find applications in many areas, including economics \cite{Jorgensen:2010}, cyber-security \cite{Etesami:2019}, communication networks \cite{Zazo:2016}, robotics \cite{Li:2019}, and power systems \cite{Chen:2014}.

A key limitation of classical dynamic game formulations is the assumption of rectangular action sets, in which each player selects actions from an admissible set that is independent of the actions of the other players. Consequently, coupling arises only through interdependent objective functions and system dynamics. In many applications, however, feasibility is jointly determined by limitations on shared resources, communication constraints, or aggregate constraints, resulting in coupled or non-rectangular feasible action sets. This gives rise to the generalized Nash equilibrium (GNE) framework, which extends the classical Nash equilibrium by incorporating coupled constraints into the strategic optimization of players \cite{Arrow:1954,Rosen:1965}. In static games, the existence and properties of GNE have been extensively studied using variational inequality techniques \cite{Harker:1991,Facchinei:2010_a}. By contrast, analogous results in dynamic game settings remain comparatively underdeveloped.

For unconstrained finite and infinite-horizon dynamic games, a relatively complete theory exists; see, e.g., \cite{Basar:1999,Engwerda:2005,Engwerda:2000,Engwerda:2007,Monti:2024}. In contrast, constrained dynamic games have received comparatively less attention. Linear-quadratic (LQ) differential games
with implicit equality constraints modeled by differential-algebraic equations have been studied in \cite{Engwerda:2009,Engwerda:2012}, and stochastic variants are analyzed in \cite{Tanwani:2019}. The authors in \cite{Reddy:2015,Reddy:2017} studied LQ difference games with affine constraints, under structural decompositions and partial decoupling assumptions. More general formulations with equality and inequality constraints have been considered in \cite{Laine:2023}, together with numerical methods for feedback Nash equilibria. Dynamic potential games with inequality constraints were studied in \cite{Zazo:2016}, where restrictive structural assumptions enabled reformulation as constrained optimal control problems. GNE for deterministic and mean-field-type constrained LQ games with finite horizons and open-loop information structures were investigated in \cite{Partha:2023} and \cite{Partha:2026a}, while quasi-hierarchical difference games with constraints were studied in \cite{Partha:2026b}. Recent works \cite{Benenati:2026,Baghbadorani:2026} have also considered receding-horizon implementations of equilibrium using variational inequality methods.

Nevertheless, a general existence theory for open-loop generalized Nash equilibria in infinite-horizon linear-quadratic deterministic difference games with coupled affine inequality constraints remains unavailable. In particular, necessary and sufficient conditions for the existence of equilibrium in this setting have not been established. This note addresses this gap in the literature on dynamic games.

\textit{Contributions:}
In this note, we study infinite-horizon deterministic linear–quadratic difference games with coupled affine inequality constraints. The contributions of this technical note are twofold.

First, we derive necessary conditions for the existence of open-loop generalized Nash equilibrium (OL-GNE) strategies for this class of infinite-horizon constrained LQ difference games and establish their sufficiency under additional assumptions. The sufficient conditions are characterized in terms of the existence of square-summable solutions to associated infinite-horizon linear complementarity systems (LCSs); see Theorems \ref{cor:C1} and \ref{cor:C2}. 
To the best of our knowledge, this provides the first characterization of OL-GNE in infinite-horizon LQ difference games with coupled affine inequality constraints.

Second, we derive verifiable semi-analytic existence conditions for OL-GNE under additional structural assumptions. Although, the exact characterization is expressed via square-summable solutions of infinite-horizon LCSs, the proposed reformulation yields conditions that can be evaluated directly from the problem data; see Theorems \ref{th:Th_K} and \ref{th:LCP}. In addition, these reformulations also enable the computation of OL-GNE strategies, providing an alternative to the receding-horizon MPC-based computational approaches in \cite{Benenati:2026,Baghbadorani:2026}.

This note is organized as follows. Section \ref{sec:Preliminaries}, introduces infinite-horizon LQ difference games with coupled-affine inequality constraints. Necessary and sufficient conditions for OL-GNE are presented in Sections  \ref{sec:Necessary} and \ref{sec:Sufficient}, respectively. The sufficient conditions for OL-GNE are reformulated into more tractable form in Section \ref{sec:Solvability}. Section \ref{sec:Numerical} illustrates our results using a vehicle platooning example with coupled constraints and Section \ref{sec:Conclusions} concludes this note.

\textit{Notation:} 
We denote the sets of natural numbers, non-negative integers, real numbers, $n$-dimensional Euclidean space, $n$-dimensional non-negative orthant, and $n \times m$ real matrices by $\mathbb{N}$, $\mathbb{N}_{0}$, $\mathbb{R}$, $\mathbb{R}^n$, $\mathbb{R}^n_{+}$, and $\mathbb{R}^{n \times m}$, respectively. Transposes of a vector $a$ and a matrix $A$ are denoted by $a^{\prime}$ and $A^{\prime}$. 
For $A \in \mathbb{R}^{n \times n}$ and $a \in \mathbb{R}^n$, with $n = n_1 + \cdots + n_K$, $[A]_{ij}$ and \([a]_i\) 
denote the corresponding $n_j\times n_j$ submatrix and $n_i\times 1$ subvector, respectively.
Column vectors $[v_1^{\prime}, \cdots, v_n^{\prime}]^{\prime}$ are written as $\col{v_1, \cdots, v_n}$ or $\col{v_k}_{k=1}^n$. 
$\ell^2(\R^{n})$ is the space of all square-summable sequences taking values in $\R^n$. For a sequence $(x_k)_{k=0}^{\infty}$, $x_k\to 0$ denotes $\underset{k\to \infty}{\text{lim}}x_k=0$. The symbols $I$ and $0$ denote the identity and zero matrices of compatible dimensions, respectively.
The block diagonal matrix with diagonal elements $M_1, \cdots, M_K$ is denoted by $\oplus_{k=1}^{K}M_k$. The Kronecker product is $\otimes$. 
Vectors $x, y \in \mathbb{R}^n$ are complementary if $x \geq 0$, $y \geq 0$, and $x^{\prime}y = 0$, denoted by $0 \leq x \perp y \geq 0$.
\section{Infinite-Horizon Difference Games with Coupled Inequality Constraints}\label{sec:Preliminaries}
In this section, we introduce the infinite-horizon non-zero sum LQ difference game with coupled inequality constraints. We consider $N\geq 2$ players with $N\in\mathbb{N}$ and denote the set of players by $\N = \{1, 2, \cdots, N\}$. At each time instant $k \in \mathbb{N}_0$, each player $i \in \N$ selects an action $u^i_k\in \R^{m_i}$ and influences the evolution of the state as follows:
\begin{subequations}\label{eq:Game1}
    \begin{align}
        &x_{k+1}=Ax_k+\sum_{j\in\N}B^ju_k^j=Ax_k+\bar{B}\bu_k,\label{eq:Gstate}
    \end{align}
    where $A \in \R^{n\times n}, B^{i}\in \R^{n\times m_i}$, $\bar{B}=[B^{1}~B^{2} \cdots B^{N}]$, $\bu_k:=\col{u_k^{i}}_{i=1}^{N}\in \R^m$ ($m=\sum_{i\in\N}m_i$), with a given initial condition $x_0\in\R^n$. We further assume that, at every $k \in \mathbb{N}_0$, these decision variables for each player $i\in\N$ satisfy the following mixed coupled-affine inequality constraints
    \begin{align}
        &M^ix_k+\sum_{j\in\N}N^{ij}u_k^{j}+r^i \geq 0,\label{eq:Gconstraints}
    \end{align}
    where $M^{i} \in \R^{c_i\times n}, N^{ij}\in \R^{c_i\times m_j}, r_k^{i}\in \R^{c_i}$.
    For player $i\in\N$, we denote the rest all of player except player $i$ as $-i$, i.e., $-i:=\N\setminus\{i\}$.
	At any instant $k\in \mathbb{N}_0$ the collection of actions of all players except player $i$ is denoted by $u^{-i}_k:=\col{u_k^1,\cdots,u_k^{i-1},u_k^{i+1},\cdots,u_k^N}$. The profile of actions, also referred to as a strategy, of player $i\in \N$ is denoted by $\bu^{i}:=(u_k^{i})_{k=0}^{\infty}$, and the  strategies of all players except player $i$ is denoted by $\bu^{-i}:=(u_k^{-i})_{k=0}^{\infty}$. Each player $i\in \N$, using her strategy $\mathsf {u}^i$, seeks to minimize the following interdependent stage-additive infinite-horizon cost functional
    \begin{align}
        &J^i(x_0,(\bu^i,\bu^{-i}))=\tfrac{1}{2}\sum_{k=0}^{\infty}\big(x_k^{\prime}Q^ix_k+\sum_{j\in\N}u_k^{j\prime}R^{ij}u_k^j\big),\label{eq:Gcost}
    \end{align}
\end{subequations}
where $R^{i j}\in  {\R^{m_i \times m_j}}$, $R^{i i}=R^{ii\prime}$, $Q^i\in \R^{n \times n},\, Q^i =Q^{i\prime}$ for $i,j \in \N$.  
Due to linear dynamics, coupled constraints and interdependent infinite-horizon quadratic objectives, \eqref{eq:Game1} constitutes a $N$-player infinite-horizon  non-zero-sum LQ difference game with coupled inequality constraints, which we refer to as IDGC for the  remainder of this note.
\section{Open-loop Generalized Nash Equilibrium}
In this section, we derive necessary and sufficient conditions for the existence of open-loop generalized Nash equilibrium for the IDGC \eqref{eq:Game1}. We begin by defining admissible strategy spaces for the players and state the required assumptions.
Admissibility of the players’ actions is determined by the coupled inequality constraints in \eqref{eq:Gconstraints}. Moreover, we restrict our attention to stabilizing control strategies that drive the state to the origin while ensuring the cost functional in \eqref{eq:Gcost} is finite. 
By eliminating the state variable from \eqref{eq:Gconstraints} using \eqref{eq:Gstate}, and collecting the constraints of all players, we obtain the joint constraints at stage $k\in \mathbb{N}_0$ as follows:
\begin{align}
	&\bar{M} (A^{k} x_0+A^{k-1}B\bu_{0}+\cdots
	+ AB\bu_{k-2}+ B\bu_{k-1})\nonumber\\
	&\hspace{1.5in}+\bar{N} \bu_k+r \geq 0,\label{eq:vectorC1}
\end{align}
where, $\bar{M}=\col{M^i}_{i=1}^{N}\in\R^{c\times n}$, $\bar{N}=\col{[N^{i1} \cdots N^{iN}]}_{i=1}^{N}\in\R^{c\times m}$,  $r=\col{r^i}_{i=1}^N\in\R^{c}$ and $c=\sum_{i\in\N}c_i$. We define the set
\begin{align}
	\Omega:&=\{(x_0,(\bu^i,\bu^{-i})) \in \mathbb R^n \times \ell^2(\R^{m})\nonumber\\
    &\hspace{1.0in}~|~\eqref{eq:vectorC1} \text{ holds for all } k\in \mathbb{N}_0\}.
	\label{eq:sigma}
\end{align}
The set of initial conditions for which all coupled constraints in \eqref{eq:vectorC1} hold is then given by
\begin{align}
	\mathsf X_0:=\{x_0\in \mathbb R^n~|~\Omega \neq \emptyset\}. \label{eq:feasiableinit}
\end{align}
Clearly,  $\Omega\neq \emptyset$ implies $\mathsf X_0\neq \emptyset$. For any $x_0\in \mathsf X_0$, the joint feasible strategy space of the players is given by
\begin{align}
	\mathsf{R}(x_0):=\{(\bu^{i}, \bu^{-i})\in \ell^2(\R^{m})~|~(x_0,(\bu^i,\bu^{-i}))\in \Omega\}.
	\label{eq:CConstraint}
\end{align} 
For any $x_0\in \mathsf X_0$, using the set $\mathsf{R}(x_0)$, we define the following:
\begin{subequations}
    \begin{align}
	&\U^{-i}(x_0):=\{\bu^{-i}\in  \Uppi_{j=\in-i} \ell^2(\R^{m_j})~|~ \exists \bu^i \in \ell^2(\R^{m_i})\nonumber\\
    &\hspace{1.0in}~\text{such that}~ (\bu^{i}, \bu^{-i}) \in \mathsf{R}(x_0)\}.\label{eq:Uj}
\end{align}
For a given $x_0\in\X_0$, $\U^{-i}(x_0)$ is the set of all control sequences of players in $\N\setminus\{i\}$, for which player $i$ has control sequences satisfying the coupled constraints \eqref{eq:Gconstraints}. Next, for a given $\bu^{-i}\in \U^{-i}(x_0)$, we define the admissible action space of player $i$ as follows:
\begin{align}
    &\U^i(\bu^{-i};x_0):=\{\bu^{i}\in \ell^2(\R^{m_i})~|~(\bu^{i}, \bu^{-i}) \in \mathsf{R}(x_0),\nonumber\\
    &\hspace{1.0in}~J^i(x_0,(\bu^i,\bu^{-i})) <\infty, ~x_k \to 0\}.\label{eq:AdmissibleSet}
\end{align}
\end{subequations}
From the above, note that, for a given $\bu^{-i}\in \U^{-i}(x_0)$, the set $\U^i(\bu^{-i};x_0)$ is the collection of all square-summable stabilizing controls of player $i$ that yield a finite cost. Next, for IDGC \eqref{eq:Game1}, we make the following standard assumptions:
\begin{assumption}\label{ass:G1}
$\Omega\neq \emptyset$ and for $i\in\N$, (i) $\U^i(\bu^{-i};x_0)$ is nonempty for $x_0\in \mathsf X_0$, (ii) $Q^i\succeq 0$, $R^{ii} \succ 0$, (iii) $(A, B^i)$ is stabilizable, $(A, Q^i)$ is detectable and (iv) no row of $N^{ii}$ is identically zero.
\end{assumption}
\noindent 
If \(\Omega \neq \emptyset\), then from \eqref{eq:feasiableinit} and \eqref{eq:CConstraint}, $\mathsf{R}(x_0) \neq \emptyset$ for $x_0 \in \mathsf{X}_0$. By Item (i) admissible controls for each player $i\in\N$ exist, for which her cost is finite (see \eqref{eq:AdmissibleSet}). Item (ii) and (iii) are standard technical requirements. Under Item (iv), only coupled constraints that explicitly involve player $i$'s control variables are included in \eqref{eq:Gconstraints} of player $i$.
%
\begin{remark}\label{rem:ri_A1}
    If Assumption \ref{ass:G1} holds, then Item (i) and \eqref{eq:AdmissibleSet} imply that, for every $\bu^i\in \U^i(\bu^{-i};x_0)$, we have $(\bu^{i},\bu^{-i})\in\mathsf{R}(x_0)$ and $x_k\to 0$. Moreover, by \eqref{eq:CConstraint} and \eqref{eq:sigma}, $\bu^i\in \ell^2(\R^{m_i})$, so $u_k^i\to 0$ for each $i\in\N$. Since \eqref{eq:sigma} also ensures that \eqref{eq:Gconstraints} holds $\forall k\in\mathbb{N}_0$, taking the limit as $k\to\infty$ in \eqref{eq:Gconstraints} yields $r^i\ge 0$ for all $i\in\N$. Hence, the condition $r^i\ge 0$, $i\in\N$, is implicitly contained in Assumption \ref{ass:G1}.
\end{remark}
\begin{remark}\label{rem:purestateC}
 Affine pure state constraints for $k\geq 1$ can be reformulated as mixed constraints \eqref{eq:Gconstraints} using the linear dynamics \eqref{eq:Gstate}, and thus fit within our framework, provided the state is coupled to at least one player action (i.e., corresponding row of $\bar{B}$ is not identically zero). Pure state constraints at $k = 0$ (i.e., on $x_0$) simply further restrict the sets $\Omega$ and $\X_0$.
\end{remark}
Next, we present the definition of open-loop generalized Nash equilibrium for IDGC \eqref{eq:Game1}.
\begin{definition}\label{def:OCNEdef} 
 For a given $x_0\in\X_0$, an admissible strategy profile $\bu^{\star}=\col{\bu^{i\star}}_{i=1}^{N}$ with $\bu^{i\star}\in 	\U^i(\bu^{-i\star}; x_0),~ i\in\N$, is an open-loop generalized Nash equilibrium (OL-GNE) for IDGC if for each $i\in \N$ the following inequality holds 
	\begin{align}
		J^i(x_0, (\bu^{i\star}, \bu^{-i\star})) \leq J^i(x_0, (\bu^i, \bu^{-i\star})),~\forall \bu^i\in \U^i(\bu^{-i\star}; x_0).\label{eq:OCNEdef}
	\end{align} 
\end{definition}
\subsection{Necessary Conditions}\label{sec:Necessary}
The OL-GNE of IDGC \eqref{eq:Game1} is obtained by simultaneously solving the $N$ coupled constrained infinite-horizon optimal control problems in \eqref{eq:OCNEdef}, where, $\forall \bu^i\in\U^i(\bu^{-i\star};x_0)$, the inequalities $M^ix_k+N^{ii}u_k^i+\sum_{j\in-i}N^{ij}u_k^{j\star}+r^i\geq 0,~k\in\mathbb{N}_0$ hold.
The necessary conditions for the existence of OL-GNE in IDGC are obtained by applying the infinite-horizon discrete time maximum principle \cite{Blot:2014} to each of these problems in \eqref{eq:OCNEdef} and are given by the following infinite-horizon discrete-time coupled linear complementarity system (LCS):
\begin{subequations}\label{eq:LCS1}
    \begin{align}
        &x_{k+1}^{\star}=Ax_k^{\star}-\sum_{j\in\N}B^j(R^{jj})^{-1}(B^{j\prime}\lambda_{k+1}^{j\star}-{N^{jj}}^{\prime}\mu_k^{j\star}),\label{eq:LCS1_eq1}\\
        &\lambda_k^{i\star}=Q^ix_k^{\star}+A^{\prime}\lambda_{k+1}^{i\star}-{M^i}^{\prime}\mu_k^{i\star},~i\in \N,\label{eq:LCS1_eq2}\\
        &0\leq M^ix_k^{\star}-\sum_{j\in\N}N^{ij}(R^{jj})^{-1}(B^{j\prime}\lambda_{k+1}^{j\star}-{N^{jj}}^{\prime}\mu_k^{j\star})\nonumber\\
        &\hspace{1.25in}+r^i \perp \mu_k^{i\star}\geq 0,~i\in \N,\label{eq:LCS1_eq3}
    \end{align}
\end{subequations}
   with $x_0^{\star}=x_0$. If the LCS \eqref{eq:LCS1} admits a solution and $R^{ii}\succ 0$, then the corresponding candidate equilibrium control of player $i\in\N$ is
\begin{align}
    u_k^{i\star}=-(R^{ii})^{-1}\big(B^{i\prime}\lambda_{k+1}^{i\star}-{N^{ii}}^{\prime}\mu_k^{i\star}\big), ~k\in\mathbb{N}_0.\label{eq:BarU1}
\end{align}
Since \eqref{eq:LCS1} provides only the necessary conditions, any solution, if it exists, yields a candidate OL-GNE for \eqref{eq:Game1}. Furthermore, due to the stabilizing property of an OL-GNE, the associated costate sequences $\lambda_k^{i\star}$, $i\in\N$, must additionally satisfy suitable transversality conditions as $k\to \infty$ (see \cite{Aseev:2017, Blot:2014} for various transversality conditions). The specific transversality conditions
imposed on the solutions of LCS \eqref{eq:LCS1}, will be introduced later in the discussion of sufficient conditions.


\subsection{Sufficient Conditions}\label{sec:Sufficient}
In this section, we derive sufficient conditions for existence of OL-GNE in IDGC \eqref{eq:Game1}. To this end, we make the following assumption.
\begin{assumption}\label{ass:G2}
    For each $i\in\N$, let $E^i\succ 0$ be the stabilizing solution to the following discrete-time Riccati equation
     \begin{align}
         E^i=Q^i+A^{\prime}E^iA-A^{\prime}E^{i}B^i(R^{ii}+B^{i\prime}E^iB^i)^{-1}B^{i\prime}E^iA.\label{eq:Eeq}
     \end{align}
\end{assumption}
\begin{remark}\label{rem:R1}
    Under Assumptions \ref{ass:G1}.(ii) and \ref{ass:G2}, note that $Y^i=R^{ii}+B^{i\prime}E^iB^i \succ 0$. Further, from \eqref{eq:Eeq}, $(A-B^iL^i)^{\prime}E^i (A-B^iL^i)-E^i=-Q^i-L^{i\prime}R^{ii}L^i \preceq 0$, where $L^i=(Y^i)^{-1}B^{i\prime}E^iA$.
    So, by Assumptions \ref{ass:G2}, $A-B^iL^i=(I-B^i(Y^i)^{-1}B^{i\prime}E^i)A$ is Schur stable.
\end{remark}
Next, we present four lemmas that will be used to establish sufficient conditions for the existence of OL-GNE in IDGC.
\begin{lemma}\label{lem:L00}
Suppose $Z\in\R^{n\times n}$ is Schur stable and $z_0$ is finite. (i) If $z_{k+1}=Zz_{k}$, then $(z_k)_{k=0}^{\infty}\in \ell^2(\R^{n})$. (ii) If $z_{k+1}=Zz_{k}+Ss_k$ and $(s_k)_{k=0}^{\infty}\in \ell^2(\R^{m_i})$, then $(z_k)_{k=0}^{\infty}\in \ell^2(\R^{n})$.
\end{lemma}
\begin{proof}
    (i) As $Z$ is Schur stable, there exist some $0<C<\infty$ and $0 \leq \rho <1$, such that $\normx{Z^{\tau}}\leq C\rho^{\tau}$. So  $\normx{z_k}\leq\normx{Z^{k}}\normx{z_{0}}\leq C\rho^k\normx{z_0}$ and this implies $\sum_{k=0}^{\infty}\normx{z_k}^2\leq C^2\normx{z_0}^2\sum_{k=0}^{\infty}\rho^{2k}=\tfrac{C^2\normx{z_0}^2}{(1-\rho^2)}<\infty$ i.e., $(z_k)_{k=0}^{\infty}\in \ell^2(\R^{n})$.

    (ii) Note that if $z_{k+1}=Zz_{k}+Ss_k$, then $z_k=Z^kz_0+\sum_{\tau=0}^{k-1}Z^{k-1-\tau}Ss_{\tau}$. Since $Z$ is Schur stable, we have $\normx{z_k}\leq C\rho^k\normx{z_0}+C\normx{S}\sum_{\tau=0}^{k-1}\rho^{k-1-\tau}\normx{s_\tau}$. Squaring both sides and using the inequality $(a+b)^2\leq 2(a^2+b^2)$, we have $\normx{z_k}^2\leq 2C^2\normx{z_0}^2\rho^{2k}+2C^2\normx{S}^2\big(\sum_{\tau=0}^{k-1}\rho^{k-1-\tau}\normx{s_\tau}\big)^2$. But, by Cauchy–Schwarz inequality
    \begin{align*}
        &\big(\sum_{\tau=0}^{k-1}\rho^{k-1-\tau}\normx{s_\tau}\big)^2\leq \big(\sum_{\tau=0}^{k-1}\rho^{k-1-\tau}\big)\big(\sum_{\tau=0}^{k-1}\rho^{k-1-\tau}\normx{s_\tau}^2\big)\\
        &\leq \tfrac{1}{(1-\rho)}\sum_{\tau=0}^{k-1}\rho^{k-1-\tau}\normx{s_\tau}^2~(\text{as}~\sum_{\tau=0}^{k-1}\rho^{k-1-\tau}\leq \tfrac{1}{(1-\rho)}).
    \end{align*} 
    Therefore, we have
    \begin{align*}
        \sum_{k=0}^{\infty}\normx{z_k}^2&\leq  2C^2\normx{z_0}^2\sum_{k=0}^{\infty}\rho^{2k}+\tfrac{2C^2\normx{S}^2}{(1-\rho)}\sum_{k=0}^{\infty}\sum_{\tau=0}^{k-1}\rho^{k-1-\tau}\normx{s_\tau}^2\\
        &=  \tfrac{2C^2\normx{z_0}^2}{(1-\rho^2)}+\tfrac{2C^2\normx{S}^2}{(1-\rho)}\sum_{\tau=0}^{\infty}\normx{s_\tau}^2\sum_{k=\tau+1}^{\infty}\rho^{k-1-\tau}\\
        &= \tfrac{2C^2\normx{z_0}^2}{(1-\rho^2)}+\tfrac{2C^2\normx{S}^2}{(1-\rho)^2}\sum_{\tau=0}^{\infty}\normx{s_\tau}^2<\infty
    \end{align*}
    The last inequality follows as $(s_k)_{k=0}^{\infty}\in \ell^2(\R^{m_i})$. Therefore, $(z_k)_{k=0}^{\infty}\in \ell^2(\R^{n})$.
\end{proof}
\begin{lemma}\label{lem:L0}
Let $z_k=Zz_{k+1}+Ss_k$, $z_k \to 0$, $Z\in\R^{n\times n}$ is Schur stable and $(s_k)_{k=0}^{\infty}\in \ell^2(\R^{m_i})$, then $(z_k)_{k=0}^{\infty}\in \ell^2(\R^{n})$.
\end{lemma}
\begin{proof}
    Due to Schur stability of $Z$, there exist some $0<C<\infty$ and $0 \leq \rho <1$, such that $\normx{Z^{\tau}}\leq C\rho^{\tau}$. As $z_k \to 0$ and $Z$ is Schur stable, we have $z_k= \sum_{\tau=0}^{\infty}Z^{\tau}Ss_{k+\tau}$ (the series on the right-hand side is well-defined because it converges to $z_k$) and thus $\normx{z_k}\leq \sum_{\tau=0}^{\infty}\normx{Z^\tau}\normx{S}\normx{s_{k+\tau}}$ $\leq C\normx{S}\sum_{\tau=0}^{\infty}\rho^{\tau}\normx{s_{k+\tau}}$ (using $\normx{Z^{\tau}}\leq C\rho^{\tau}$). But, by Cauchy–Schwarz inequality $(\sum_{\tau=0}^{\infty}\rho^{\tau}\normx{s_{k+\tau}})^2\leq(\sum_{\tau=0}^{\infty}\rho^{\tau})(\sum_{\tau=0}^{\infty}\rho^{\tau}\normx{s_{k+\tau}}^2)=\tfrac{1}{(1-\rho)}(\sum_{\tau=0}^{\infty}\rho^{\tau}\normx{s_{k+\tau}}^2)$. Therefore, we have
    \begin{align*}
        \sum_{k=0}^{\infty}\normx{z_k}^2&\leq \tfrac{C^2\normx{S}^2}{(1-\rho)}\sum_{k=0}^{\infty}\sum_{\tau=0}^{\infty}\rho^{\tau}\normx{s_{k+\tau}}^2\\
        &=\tfrac{C^2\normx{S}^2}{(1-\rho)}\sum_{\tau=0}^{\infty}\rho^{\tau}\sum_{k=0}^{\infty}\normx{s_{k+\tau}}^2\\
        &\leq\tfrac{C^2\normx{S}^2}{(1-\rho)}\sum_{\tau=0}^{\infty}\rho^{\tau}\sum_{k=0}^{\infty}\normx{s_{k}}^2\\
        &=\tfrac{C^2\normx{S}^2}{(1-\rho)^2}\sum_{k=0}^{\infty}\normx{s_{k}}^2< \infty~(\text{as}~(s_k)_{k=0}^{\infty}\in \ell^2(\R^{m_i}))
    \end{align*}
    So $(z_k)_{k=0}^{\infty}\in \ell^2(\R^{n})$.
\end{proof}
\begin{lemma}\label{lem:L1}
   Let Assumptions \ref{ass:G1} and \ref{ass:G2} hold. Let $x_0\in\X_0$ and for a given $\bu^{-i}\in\U^{-i}(x_0)$ there exist multipliers $(\mu_k^i)_{k=0}^{\infty}\in \ell^2(\R^{c_i}_{+})$ satisfying the following discrete time linear complementarity system 
\begin{subequations}\label{eq:LCS_1}
\begin{align}
    &0 \leq (M^i-N^{ii}L^i)x_k-N^{ii}b_k^i+v_k^{i}+r^i\perp \mu_k^{i}\geq 0, \label{eq:MU1}\\
    &x_{k+1}=(A-B^iL^i)x_k-B^ib_k^i+w_k^{i}, \label{eq:X1}\\
    &e_k^i=A^{\prime}(e_{k+1}^{i}+E^iw_k^{i})-A^{\prime}E^{i}B^ib_k^i-M^{i\prime}\mu_k^{i},\label{eq:eeq}\\
    &Y^ib_k^i=B^{i\prime}(e_{k+1}^i+E^iw_k^{i})-{N^{ii}}^{\prime}\mu_k^{i},\label{eq:beq}
\end{align}
\end{subequations}
for all $k\in\mathbb{N}_0$ and $e_{k}^i \to 0$, where $L^i=(Y^i)^{-1}B^{i\prime}E^iA$, $w_k^{i}=\sum_{j\in-i}B^ju_k^{j}$ and $v_k^{i}=\sum_{j\in-i}N^{ij}u_k^{j}$. If $\bar{\bu}^{i}=(\bar{u}_k^{i})_{k=0}^{\infty}$, with
 \begin{align}
    \bar{u}_k^{i}=-L^ix_k-b_k^i.\label{eq:U1}
\end{align}
then $(e_k^{i})_{k=0}^{\infty}\in \ell^2(\R^{n})$,  $(x_k)_{k=0}^{\infty}\in \ell^2(\R^n)$, $(\bar{u}_k^{i})_{k=0}^{\infty}\in \ell^2(\R^{m_i})$.
\end{lemma}
\begin{proof}
    Since $\bu^{-i}\in\U^{-i}(x_0)$, from \eqref{eq:Uj}, we have $\bu^j\in \ell^2(\R^{m_j})$, $\forall j\in\N\setminus\{i\}$ and $w_k^i \to 0$ and $v_k^i \to 0$. In addition, $e_{k}^i \to 0$ and $\mu_k^{i} \to 0$ (as $(\mu_k^{i})_{k=0}^{\infty}\in \ell^2(\R^{c_i}_{+})$). So, from \eqref{eq:beq}, we observe that $b_k^i \to 0$ (as $Y^i\succ 0$ by Remark \ref{rem:R1}). From Remark \ref{rem:R1}, the matrix $A-B^iL^i$ is Schur stable. Thus, from \eqref{eq:X1}, for any $x_0 \in \X_0$, we have $x_k \to 0$ as $b_k^i \to 0$ and $w_k^i \to 0$.

    Next, we show $(e_k^{i})_{k=0}^{\infty}\in \ell^2(\R^{n})$. For this, using the definition \eqref{eq:beq} of $b_k^i$, we rewrite \eqref{eq:eeq} as $e_k^i=(A-B^iL^i)^{\prime}e_{k+1}^{i}+s_k^i$, where $s_k^i=(A-B^iL^i)^{\prime}E^iw_k^{i}+(N^{ii}L^{i}-M^i)^{\prime}\mu_k^i$. Then, from Lemma \ref{lem:L0}, we have $(e_k^{i})_{k=0}^{\infty}\in \ell^2(\R^{n})$, as $A-B^iL^i$ is Schur stable (see Remark \ref{rem:R1}), $(s_k^{i})_{k=0}^{\infty}\in \ell^2(\R^n)$ and $e_{k}^i \to 0$. From \eqref{eq:beq}, this also implies $(b_k^{i})_{k=0}^{\infty}\in \ell^2(\R^{m_i})$. Furthermore, using Lemma \ref{lem:L00}.(ii), from \eqref{eq:X1}, we also have $(x_k)_{k=0}^{\infty}\in \ell^2(\R^n)$ as $A-B^iL^i$ Schur stable and $(b_k^{i})_{k=0}^{\infty}\in \ell^2(\R^{m_i})$, $(w_k^{i})_{k=0}^{\infty}\in \ell^2(\R^n)$. Finally, from the definition of $\bar{u}_k^{i}$ in \eqref{eq:U1}, $(\bar{u}_k^{i})_{k=0}^{\infty}\in \ell^2(\R^{m_i})$ as both $(x_k)_{k=0}^{\infty}\in \ell^2(\R^n)$ and $(b_k^{i})_{k=0}^{\infty}\in \ell^2(\R^{m_i})$.
\end{proof}
\begin{lemma}\label{lem:finite_sum}
    If the hypothesis of Lemma \ref{lem:L1} hold, then for any admissible control $\bu^i\in \U^i(\bu^{-i};x_0)$, we have $\sum_{k=0}^{\infty}{\mu_k^{i}}^{\prime}\big(M^ix_k+N^{ii}u_k^{i}+v_k^{i}+r^i\big)<\infty$, where the state evolves as $x_{k+1}=Ax_k+B^iu_k^i+w_k^i$.
\end{lemma}
\begin{proof}
First, we show $(x_k)_{k=0}^{\infty}\in \ell^2(\R^{n})$. For any $\bu^i\in \U^i(\bu^{-i};x_0)$, \eqref{eq:AdmissibleSet} implies that $\bu^j\in \ell^2(\R^{m_j})$, $\forall j\in\N$ and $J^i(x_0,(\bu^i,\bu^{-i})) <\infty$. Hence, by Assumption \ref{ass:G1}.(ii), $\sum_{k=0}^{\infty}x_k^{\prime}Q^ix_k<\infty$ i.e., $((Q^i)^{1/2}x_k)_{k=0}^{\infty}\in \ell^2(\R^{n})$. Moreover, $(w_k^i)_{k=0}^{\infty}\in \ell^2(\R^{n})$ as $w_k^{i}=\sum_{j\in-i}B^ju_k^{j}$. By detectability of $(A,Q^i)$, there exists a matrix $F\in\R^{n\times n}$, such that $A-F(Q^i)^{1/2}$ is Schur stable. Next, we express the state dynamics as $x_{k+1}=(A-F(Q^i)^{1/2})x_k+s_k$ where $s_k=F(Q^i)^{1/2}x_k+B^iu_k^i+w_k^i$. From the above, note that $(s_k)_{k=0}^{\infty}\in \ell^2(\R^{n})$. So by Lemma \ref{lem:L00}.(ii), $(x_k)_{k=0}^{\infty}\in \ell^2(\R^{n})$. 

Next, define $y_k=M^ix_k+N^{ii}u_k^{i}$. Then, from the above, $(y_k)_{k=0}^{\infty}\in \ell^2(\R^{c_i})$.
Moreover, from the complementarity condition \eqref{eq:MU1}, we have ${\mu_k^{i}}^{\prime}(-\bar{y}_k+v_k^{i}+r^i)=0$ or ${\mu_k^{i}}^{\prime}(v_k^{i}+r^i)={\mu_k^{i}}^{\prime}\bar{y}_k$, where $\bar{y}_k=N^{ii}b_k^i-(M^i-N^{ii}L^i)x_k$. Also, by Lemma \ref{lem:L1}, $(\mu_k^i)_{k=0}^{\infty}\in \ell^2(\R^{c_i}_{+})$ and $(\bar{y}_k)_{k=0}^{\infty}\in \ell^2(\R^{c_i})$ (note that the state $x_k$ in Lemma \ref{lem:L1}, generated by player $i$'s control \eqref{eq:U1}, differs from the state considered here, which is generated by an arbitrary control $\bu^i\in \U^i(\bu^{-i};x_0)$). 
Finally, the result follows by Cauchy–Schwarz inequality as $\sum_{k=0}^{\infty}{\mu_k^{i}}^{\prime}\big(M^ix_k+N^{ii}u_k^{i}+v_k^{i}+r^i\big)=\sum_{k=0}^{\infty}\big({\mu_k^{i}}^{\prime}y_k+{\mu_k^{i}}^{\prime}\bar{y}_k\big) \leq \sum_{k=0}^{\infty}\normx{\mu_k^{i}}\normx{y_k}+\sum_{k=0}^{\infty}\normx{\mu_k^{i}}\normx{\bar{y}_k}\leq \big(\sum_{k=0}^{\infty}\normx{\mu_k^i}^2\big)^{1/2}\big(\sum_{k=0}^{\infty}\normx{y_k}^2\big)^{1/2}+\big(\sum_{k=0}^{\infty}\normx{\mu_k^i}^2\big)^{1/2}\big(\sum_{k=0}^{\infty}\normx{\bar{y}_k}^2\big)^{1/2}<\infty$.
\end{proof}
The next theorem shows that, the controls $\bar{\bu}^{i}$ in \eqref{eq:U1}, are best responses of player $i$ to a given $\bu^{-i}\in\U^{-i}(x_0)$ of other players.
\begin{theorem}\label{thm:T1}
Let the conditions of Lemma \ref{lem:L1} hold. Then, the controls in \eqref{eq:U1} satisfy $\bar{\bu}^{i}\in \U^i(\bu^{-i}; x_0)$ and are best responses of player $i$ to a given $\bu^{-i}\in\U^{-i}(x_0)$ of all other players. Moreover, the optimal cost of player $i$ is $\tfrac{1}{2}x_0^{\prime}E^ix_0+e_0^{i\prime}x_0+f_0^i$,
where $f_k^i \to 0$ and
  \begin{align}
        f_k^i&=f_{k+1}^i+{w_k^{i}}^{\prime}e_{k+1}^{i}+\tfrac{1}{2}{w_k^{i}}^{\prime}E^iw_k^{i}+\tfrac{1}{2}\sum_{j\in-i}{u_k^{j}}^{\prime}R^{ij} u_k^{j}\nonumber\\
        &\quad-\tfrac{1}{2}b_k^{i\prime}Y^ib_k^i-{\mu_k^{i}}^{\prime}\big(v_k^{i}+r^i\big),~k\in\mathbb{N}_0. \label{eq:feq} 
  \end{align}
\end{theorem}
\begin{proof}
The proof is completed using an approach similar to \cite{Partha:2023}, which consists of the following five steps. 
\noindent\textbf{Step 1:} Define a guess functional for player $i$, as $V_k^i(x_k)=\tfrac{1}{2}x_k^{\prime}E^ix_k+e_k^{i\prime}x_k+f_k^i$.\\
\noindent\textbf{Step 2:} We compute the telescopic sum of the guess functional i.e., $\sum_{k=0}^{\infty}\big(V_{k+1}^i(x_{k+1})-V_k^i(x_k)\big)$, using the state dynamics $x_{k+1}=Ax_k+B^iu_k^i+w_k^i$ for an arbitrary admissible control $\bu^i\in \U^i(\bu^{-i};x_0)$.
Note that, we have $f_k^i \to 0$ and from Lemma \ref{lem:L1}, $e_k^i \to 0$. Also, for any admissible control $\bu^i\in \U^i(\bu^{-i};x_0)$ of player $i$ with $\bu^{-i}\in\U^{-i}(x_0)$, it follows from \eqref{eq:AdmissibleSet} that $x_k \to 0$, implying $V_k^i(x_k) \to 0$.
So, the telescopic sum converges and
\begin{align*}
    &-V_0^i(x_0)=\sum_{k=0}^{\infty}\big(V_{k+1}^i(x_{k+1})-V_k^i(x_k)\big)\\
    &=\sum_{k=0}^{\infty}\Big(\tfrac{1}{2}x_{k}^{\prime}(A^{\prime}E^iA-E^i)x_{k}+\tfrac{1}{2}{u_k^i}^{\prime}B^{i\prime}E^iB^iu_k^i+{u_k^i}^{\prime}B^{i\prime}\\
    &\times(E^i(Ax_k+w_k^{i})+e_{k+1}^i)+(A^{\prime}e_{k+1}^{i}+A^{\prime}E^iw_k^{i}-e_k^i)^{\prime}x_k\nonumber\\
    &+{w_k^{i}}^{\prime}e_{k+1}^{i}+\tfrac{1}{2}{w_k^{i}}^{\prime}E^iw_k^{i}+f_{k+1}^i-f_k^i\Big).
\end{align*}
\noindent\textbf{Step 3:} Using the expression for $V_0^i(x_0)$ derived in Step 2, we compute $J^i(x_0,(\bu^i,\bu^{-i}))-V_0^i(x_0)$ as follows:
\begin{align*}
    &J^i(x_0, (\bu^i,\bu^{-i}))-V_0^i(x_0)=\sum_{k=0}^{\infty}\Big(\tfrac{1}{2}x_{k}^{\prime}\big(Q^i+A^{\prime}E^iA-E^i\big)x_{k}\\
    &+\tfrac{1}{2}{u_k^i}^{\prime}Y^iu_k^i+{u_k^i}^{\prime}B^{i\prime}\big(E^i\big(Ax_k+w_k^{i}\big)+e_{k+1}^i\big)\nonumber\\
    &+\big(A^{\prime}e_{k+1}^{i}+A^{\prime}E^iw_k^{i}-e_k^i\big)^{\prime}x_k+\tfrac{1}{2}\sum_{j\in-i}{u_k^{j}}^{\prime}R^{ij} u_k^{j}\\
    &+\tfrac{1}{2}{w_k^{i}}^{\prime}E^iw_k^{i}+{w_k^{i}}^{\prime}e_{k+1}^{i}+f_{k+1}^i-f_k^i\Big).
\end{align*}
\noindent\textbf{Step 4:} By Lemma \ref{lem:finite_sum},  $\sum_{k=0}^{\infty}{\mu_k^{i}}^{\prime}\big(M^ix_k+N^{ii}u_k^{i}+v_k^{i}+r^i\big)<\infty$ for any $\bu^i\in \U^i(\bu^{-i};x_0)$. Thus, we can add and subtract the terms $\sum_{k=0}^{\infty}{\mu_k^{i}}^{\prime}\big(M^ix_k+N^{ii}u_k^{i}+v_k^{i}+r^i\big)$ to the above computed expression  $J^i(x_0, (\bu^i,\bu^{-i}))-V_0^i(x_0)$ to incorporate inequality constraints. Finally, as $Y^i\succ 0$ (Remark \ref{rem:R1}), we do completion of square of terms $\tfrac{1}{2}{u_k^i}^{\prime}Y^iu_k^i+{u_k^i}^{\prime}Y^i(L^ix_k+b_k^i)$ to arrive at the following expression.
\begin{align*}
    &J^i(x_0, (\bu^i,\bu^{-i}))-V_0^i(x_0)\\
    &=\tfrac{1}{2}\sum_{k=0}^{\infty}x_{k}^{\prime}\big(Q^i+A^{\prime}E^iA-A^{\prime}E^{i}B^i(Y^i)^{-1}B^{i\prime}E^iA-E^i\big)x_{k}\\
    &+\sum_{k=0}^{\infty}\big(A^{\prime}e_{k+1}^{i}+A^{\prime}E^iw_k^{i}-A^{\prime}E^{i}B^ib_k^i-M^{i\prime}\mu_k^{i}-e_k^i\big)^{\prime}x_k\\
    &+\sum_{k=0}^{\infty}\Big(f_{k+1}^i+{w_k^{i}}^{\prime}e_{k+1}^{i}+\tfrac{1}{2}{w_k^{i}}^{\prime}E^iw_k^{i}+\tfrac{1}{2}\sum_{j\in-i}{u_k^{j}}^{\prime}R^{ij} u_k^{j}\\
    &-\tfrac{1}{2}b_k^{i\prime}Y^ib_k^i-{\mu_k^{i}}^{\prime}\big(v_k^{i}+r^i\big)-f_k^i\Big)+\sum_{k=0}^{\infty}{\mu_k^{i}}^{\prime}\big(M^ix_k+N^{ii}u_k^{i}\\
    &+v_k^{i}+r^i\big)+\tfrac{1}{2}\sum_{k=0}^{\infty}\big(u_k^i+L^ix_k+b_k^i\big)^{\prime}Y^i\big(u_k^i+L^ix_k+b_k^i\big).
\end{align*}
Due to \eqref{eq:Eeq}, \eqref{eq:eeq} and \eqref{eq:feq}, the first three summations on the right hand side of the above expression vanish. Thus, we have
\begin{align}
    &J^i(x_0, (\bu^i,\bu^{-i}))=V_0^i(x_0)+\sum_{k=0}^{\infty}{\mu_k^{i}}^{\prime}\big(M^ix_k+N^{ii}u_k^{i}+v_k^{i}\nonumber\\
    &+r^i\big)+\tfrac{1}{2}\sum_{k=0}^{\infty}(u_k^i+L^ix_k+b_k^i)^{\prime}Y^i(u_k^i+L^ix_k+b_k^i).\label{eq:J_cost1}
\end{align}
\noindent\textbf{Step 5:}  
Substituting $\bar{u}_k^{i}$ given by \eqref{eq:U1} for $u_k^i$ in \eqref{eq:J_cost1}, the last summation on the right hand side of \eqref{eq:J_cost1} vanishes. Also, by the complementarity conditions \eqref{eq:MU1}, the second summation  in \eqref{eq:J_cost1} is zero for $\bar{u}_k^{i}$. So, with $\bar{\bu}^{i}=(\bar{u}_k^{i})_{k=0}^{\infty}$, we have
\begin{align}
    J^i(x_0, (\bar{\bu}^{i},\bu^{-i}))&=V_0^i(x_0).\label{eq:J_cost2}
\end{align}
By Remark \ref{rem:R1}, $Y^i\succ 0$ and by \eqref{eq:MU1}, $\mu_k^i\geq 0$. So, for all admissible $u_k^i$ i.e., $\forall \bu^i\in 	\U^i(\bu^{-i}; x_0)$, both summations in \eqref{eq:J_cost1} are non-negative (as $M^ix_k+N^{ii}u_k^{i}+v_k^{i}+r^i\geq 0$ for all admissible $u_k^i$). Thus, comparing \eqref{eq:J_cost1} and \eqref{eq:J_cost2}, for all $\bu^i\in 	\U^i(\bu^{-i}; x_0)$, we have
\begin{align}
    J^i(x_0, (\bar{\bu}^{i},\bu^{-i}))\leq  J^i(x_0, (\bu^i,\bu^{-i})).\label{eq:Th1GNE}
\end{align} 
Then the only thing left to show is $\bar{\bu}^{i}\in \U^i(\bu^{-i}; x_0)$. By Lemma \ref{lem:L1}, $(\bar{u}_k^{i})_{k=0}^{\infty}\in \ell^2(\R^{m_i})$ and $x_k\to 0$. Also, by the first inequality in \eqref{eq:MU1}, $(\bar{\bu}^{i},\bu^{-i})\in \mathsf{R}(x_0)$. In addition, from \eqref{eq:AdmissibleSet}, $J^i(x_0, (\bu^i,\bu^{-i})) < \infty$, $\forall \bu^i\in \U^i(\bu^{-i}; x_0)$. So, from \eqref{eq:Th1GNE}, $J^i(x_0, (\bar{\bu}^{i},\bu^{-i}))< \infty$. Thus, by \eqref{eq:AdmissibleSet}, $\bar{\bu}^{i}\in \U^i(\bu^{-i}; x_0)$.
Since  \eqref{eq:Th1GNE} holds for $\forall \bu^i\in 	\U^i(\bu^{-i}; x_0)$,
$\bar{\bu}^{i}$ in \eqref{eq:U1}, are best response to given $\bu^{-i}\in\U^{-i}(x_0)$ with optimal cost (from \eqref{eq:J_cost2}) $V_0^i(x_0)=\tfrac{1}{2}x_0^{\prime}E^ix_0+e_0^{i\prime}x_0+f_0^i$.
\end{proof}
Next, we
present a sufficient condition for the existence of OL-GNE in IDGC \eqref{eq:Game1}.
\begin{theorem}\label{cor:C1}
    Let Assumptions \ref{ass:G1} and \ref{ass:G2} hold and for a given $x_0\in\X_0$, the following discrete time coupled linear complementarity system
    \begin{subequations}\label{eq:LCS2}
        \begin{align}
    &\bar{x}_{k+1}=(A-\sum_{j\in\N}B^jL^j)\bar{x}_k-\sum_{j\in\N}B^j\bar{b}_k^{j},~\bar{x}_k = x_0,\label{eq:LCS2_E1}\\
    &\bar{e}_k^{i}=A^{\prime}\Big(\bar{e}_{k+1}^{i}-E^i\sum_{j\in-i}B^j(L^j\bar{x}_k+\bar{b}_k^{j})\Big)\nonumber\\
    &\hspace{0.75in}-A^{\prime}E^{i}B^i\bar{b}_k^{i}-M^{i\prime}\bar{\mu}_k^{i},\label{eq:LCS2_E2}\\
    &Y^{i}\bar{b}_k^{i}=B^{i\prime}\Big(\bar{e}_{k+1}^{i}-E^i\sum_{j\in-i}B^j(L^j\bar{x}_k+\bar{b}_k^{j})\Big)-{N^{ii}}^{\prime}\bar{\mu}_k^{i},\label{eq:bstar}\\
    &0 \leq \big(M^i-\sum_{j\in\N}N^{ij}L^j\big)\bar{x}_k-\sum_{j\in\N}N^{ij}\bar{b}_k^{j}+r^i\perp \bar{\mu}_k^{i}\geq 0, \label{eq:LCS2_E3}
    \end{align}
    \end{subequations}
  with $L^i=(Y^i)^{-1}B^{i\prime}E^iA$ is solvable for all $i\in\N$ and $k\in\mathbb{N}_0$. Further, $(\bar{x}_k)_{k=0}^{\infty}\in \ell^2(\R^n)$, $(\bar{e}_k)_{k=0}^{\infty}\in \ell^2(\R^{Nn})$ and $(\bar{\mu}_k)_{k=0}^{\infty}\in \ell^2(\R^{c}_{+})$, where $\bar{e}_k=\col{\bar{e}_k^i}_{i=1}^{N}$ and $\bar{\mu}_k=\col{\bar{\mu}_k^{i}}_{i=1}^{N}$. Then,
    \begin{align}
        \bar{u}_k^{i}&=-L^i\bar{x}_k-\bar{b}_k^{i},~~ i\in \N,~k\in\mathbb{N}_0,\label{eq:U2}
    \end{align}
   is an OL-GNE for IDGC \eqref{eq:Game1}.
\end{theorem}
\begin{proof}
     Note that $Y^i=R^{ii}+B^{i\prime}E^iB^i$, so by defining $\bar{b}_k=\col{\bar{b}_k^i}_{i=1}^{N}$, \eqref{eq:bstar}, for all $i\in\N$, can be compactly written as
     \begin{align*}
         (R+B^{\prime}EB)\bar{b}_k=B^{\prime}e_{k+1}-B^{\prime}EBL \bar{x}_k-(\otimes_{i=1}^{N}{N^{ii}})^{\prime}\bar{\mu}_k,
     \end{align*}
     where $R=\otimes_{i=1}^{N}R^{ii}$, $B=\otimes_{i=1}^{N}B^i$, $E=\otimes_{i=1}^{N}E^i$ and $L=\col{L^i}_{i=1}^{N}$. Since $R+B^{\prime}EB\succ 0$ ($R^{ii} \succ 0$ by Assumption \ref{ass:G1}.(ii) and $E^i\succ 0$ by Assumption \ref{ass:G2}) and the terms on the right of above equation are square-summable, we have $\bar{b}_k\in \ell^2(\R^{m})$. From \eqref{eq:U2}, this implies $(\bar{u}_k^{i})_{k=0}^{\infty}\in \ell^2(\R^{m_i})$ for each $i \in \N$. Further, due to the first inequality constraints in \eqref{eq:LCS2_E3},  $(\bar{\bu}^{i},\bar{\bu}^{-i})$ satisfies the coupled inequality constraints \eqref{eq:Gconstraints} for all $k\in\mathbb{N}_0$.
     Therefore, $(\bar{\bu}^{i},\bar{\bu}^{-i})\in\mathsf{R}(x_0)$ and by \eqref{eq:Uj}, we have $\bar{\bu}^{-i}\in\U^{-i}(x_0)$. Moreover, by taking $w_k^{i}=\sum_{j\in-i}B^j\bar{u}_k^{j}$ and $v_k^{i}=\sum_{j\in-i}N^{ij}\bar{u}_k^{j}$, we note that \eqref{eq:LCS2}, for player $i\in\N$, are exactly same as \eqref{eq:LCS_1} (in Lemma \ref{lem:L1}), when all players except player $i$ are using $\bar{\bu}^{-i}\in\U^{-i}(x_0)$. So by Theorem \ref{thm:T1}, $\bar{\bu}^{i}\in \U^i(\bar{\bu}^{-i}; x_0)$ and are best response of player $i$ to $\bar{\bu}^{-i}\in\U^{-i}(x_0)$. In particular, from \eqref{eq:Th1GNE}, we have 
     \begin{align}
        J^i(x_0, (\bar{\bu}^{i},\bar{\bu}^{-i}))\leq  J^i(x_0, (\bu^i,\bar{\bu}^{-i})), ~\forall \bu^i\in \U^i(\bar{\bu}^{-i}; x_0).\label{eq:GNEeqTH2}
   \end{align} 
     Since the choice of player $i\in \N$ was arbitrary, the above argument holds for all players $\forall i\in \N$ and thus, by Definition \ref{def:OCNEdef}, \eqref{eq:U2} represents an OL-GNE for IDGC \eqref{eq:Game1}.
\end{proof}
The following theorem relates LCSs \eqref{eq:LCS1} and \eqref{eq:LCS2}.
\begin{theorem}\label{th:LCSeqv}
    Let Assumptions \ref{ass:G1} and \ref{ass:G2} hold and $x_0\in\X_0$. Then the following two statements are equivalent.\\
    (i) LCS \eqref{eq:LCS1} is solvable. $\hspace{0.25in}$(ii) LCS \eqref{eq:LCS2} is solvable.
\end{theorem}
\begin{proof}
    First, we show (i)$\implies$(ii). For each $i\in\N$, using $E^i$ from \eqref{eq:Eeq}, define $e_k^{i\star}=\lambda_k^{i\star}-E^ix_k^{\star}$, $\forall k\in\mathbb{N}_0$. So, from \eqref{eq:BarU1}, $R^{ii}u_k^{i\star}=-B^{i\prime}\lambda_{k+1}^{i\star}+{N^{ii}}^{\prime}\mu_k^{i\star}=-B^{i\prime}(E^ix_{k+1}^{\star}+e_{k+1}^{i\star})+{N^{ii}}^{\prime}\mu_k^{i\star}$. Using $x_{k+1}^{\star}=Ax_k^{\star}+B^iu_k^{i\star}+\sum_{j\in-i}B^ju_k^{j\star}$ (from \eqref{eq:LCS1_eq1}) in the above expression of $R^{ii}u_k^{i\star}$ and simplifying, we have $(R^{ii}+B^{i\prime}E^iB^i)u_k^{i\star}=-B^{i\prime}E^i(Ax_k^{\star}+\sum_{j\in-i}B^ju_k^{j\star})-B^{i\prime}e_{k+1}^{i\star}+{N^{ii}}^{\prime}\mu_k^{i\star}$. Next, as $Y^i=R^{ii}+B^{i\prime}E^iB^i\succ 0$ (by Remark \ref{rem:R1}) and $L^i=(Y^i)^{-1}B^{i\prime}E^iA$, we have $u_k^{i\star}=-L^ix_k^{\star}-(Y^i)^{-1}(B^{i\prime}E^i\sum_{j\in-i}B^ju_k^{j\star}+B^{i\prime}e_{k+1}^{i\star}-{N^{ii}}^{\prime}\mu_k^{i\star})$ or $u_k^{i\star}=-L^ix_k^{\star}-b_k^{i\star}$, where, we have defined
    \begin{align}
        Y^ib_k^{i\star}=B^{i\prime}E^i\sum_{j\in-i}B^ju_k^{j\star}+B^{i\prime}e_{k+1}^{i\star}-{N^{ii}}^{\prime}\mu_k^{i\star}. \label{eq:bar_b}
    \end{align}
    So, together with \eqref{eq:BarU1}, for each $i\in\N$, $u_k^{i\star}$ can be written as
    \begin{align}
        u_k^{i\star}&=-(R^{ii})^{-1}(B^{i\prime}\lambda_{k+1}^{i\star}-{N^{ii}}^{\prime}\mu_k^{i\star})=-L^ix_k^{\star}-b_k^{i\star}.\label{eq:Bu2Bs}
    \end{align}
   Next, using \eqref{eq:Bu2Bs} in \eqref{eq:LCS1_eq1}, we have
   \begin{subequations}\label{eq:LCS_1to2}
    \begin{align}
        x_{k+1}^{\star}=(A-\sum_{j\in\N}B^jL^j)x_k^{\star}-\sum_{j\in\N}B^jb_k^{j\star},~ x_0^{\star}=x_0. \label{eq:FBeq1_1}
    \end{align}
    Further, using \eqref{eq:LCS1_eq2}, we have $e_k^{i\star}=\lambda_k^{i\star}-E^ix_k^{\star}=Q^ix_k^{\star}+A^{\prime}\lambda_{k+1}^{i\star}-{M^i}^{\prime}\mu_k^{i\star}-E^ix_k^{\star}$. Next,  using $\lambda_{k+1}^{i\star}=E^ix_{k+1}^{\star}+e_{k+1}^{i\star}$ and \eqref{eq:FBeq1_1} in the above expression of $e_k^{i\star}$ and simplifying, we have 
\begin{align*}
    &e_k^{i\star}=(Q^i+A^{\prime}E^iA-A^{\prime}E^iB^iL^i-E^i)x_k^{\star}+A^{\prime}\big(e_{k+1}^{i\star}\\
    &\quad -E^i\sum_{j\in-i}B^j(L^jx_k^{\star}+b_k^{j\star})\big)-A^{\prime}E^{i}B^ib_k^{i\star}-M^{i\prime}\mu_k^{i\star}.
\end{align*}
As $L^i=(Y^i)^{-1}B^{i\prime}E^iA$, by \eqref{eq:Eeq}, the matrix multiplied with $x_{k}^{\star}$ in the last expression vanishes. Thus, for each $i\in\N$, we have
\begin{align}
    e_k^{i\star}&=A^{\prime}\big(e_{k+1}^{i\star}-E^i\sum_{j\in-i}B^j(L^jx_k^{\star}+b_k^{j\star})\big)\nonumber\\
    &\hspace{1.0in}-A^{\prime}E^{i}B^ib_k^{i\star}-M^{i\prime}\mu_k^{i\star}.
\end{align}
  Finally, using $u_k^{j\star}=-L^jx_k^{\star}-b_k^{j\star}$ in \eqref{eq:bar_b} and \eqref{eq:LCS1_eq3}, we have
  \begin{align}
       &Y^ib_k^{i\star}=B^{i\prime}\big(e_{k+1}^{i\star}-E^i\sum_{j\in-i}B^j(L^jx_k^{\star}+b^{j\star})\big)-{N^{ii}}^{\prime}\mu_k^{i\star}, \\
       &0 \leq \big(M^i-\sum_{j\in\N}N^{ij}L^j\big)x_k^{\star}-\sum_{j\in\N}N^{ij}b_k^{j\star}+r^i\perp \mu_k^{i\star}\geq 0.
  \end{align}
\end{subequations}
for $i\in\N$. Thus, from the above, we observe that if (i) holds, then \eqref{eq:LCS_1to2} is solvable. But \eqref{eq:LCS_1to2} and \eqref{eq:LCS2} represent the same LCS, therefore LCS \eqref{eq:LCS2} is also solvable. So, (i)$\implies$(ii).

  Next, we  prove the opposite direction, i.e., (ii)$\implies$(i). Define $\bar{\lambda}_k^{i}:=E^i\bar{x}_k+\bar{e}_k^{i}$, $k\in\mathbb{N}_0$ for each $i\in\N$. As $Y^i=R^{ii}+B^{i\prime}E^iB^i \succ 0$, from \eqref{eq:U2}, $(R^{ii}+B^{i\prime}E^iB^i)\bar{u}_k^{i}=-B^{i\prime}E^iA\bar{x}_k-Y^i\bar{b}_k^{i}$. Using $Y^{i}\bar{b}_k^{i}=B^{i\prime}\big(\bar{e}_{k+1}^{i}-E^i\sum_{j\in-i}B^j(L^j\bar{x}_k+\bar{b}_k^{j})\big)-{N^{ii}}^{\prime}\bar{\mu}_k^{i}=B^{i\prime}\bar{e}_{k+1}^{i}+B^{i\prime}E^i\sum_{j\in-i}B^j\bar{u}_k^{j}-{N^{ii}}^{\prime}\bar{\mu}_k^{i}$ from \eqref{eq:bstar} and simplifying the expression for $\bar{u}_k^{i}$, we have $ R^{ii}\bar{u}_k^{i}=-B^{i\prime}E^i(A\bar{x}_k+\sum_{j\in\N}B^j\bar{u}_k^{j})-B^{i\prime}\bar{e}_{k+1}^{i}+{N^{ii}}^{\prime}\bar{\mu}_k^{i}$.
    Since $\bar{x}_{k+1}=A\bar{x}_k+\sum_{j\in\N}B^j\bar{u}_k^{j}$, we have $R^{ii}\bar{u}_k^{i}=-B^{i\prime}(E^i\bar{x}_{k+1}+\bar{e}_{k+1}^{i})+{N^{ii}}^{\prime}\bar{\mu}_k^{i}$. Finally, using $\bar{\lambda}_{k+1}^{i}=E^i\bar{x}_{k+1}+\bar{e}_{k+1}^{i}$ and $R^{ii}\succ 0$, for each $i\in \N$, we arrive at the following relation
  \begin{align}
    \bar{u}_k^{i}&=-L^i\bar{x}_k-\bar{b}_k^{i}=-(R^{ii})^{-1}\big(B^{i\prime}\bar{\lambda}_{k+1}^{i}-{N^{ii}}^{\prime}\bar{\mu}_k^{i}\big).\label{eq:U5}
 \end{align}
Next, using the above relation in \eqref{eq:LCS2_E1}, we have
\begin{subequations}\label{eq:LCS_3}
    \begin{align}
        \bar{x}_{k+1}=A\bar{x}_k-\sum_{j\in\N}B^j(R^{jj})^{-1}\big(B^{j\prime}\bar{\lambda}_{k+1}^{j}-{N^{jj}}^{\prime}\bar{\mu}_k^{j}\big), \label{eq:FBeq1}
    \end{align}
with $x_0^{\star}=x_0$. Similarly,
from \eqref{eq:LCS2_E1}, $\bar{\lambda}_{k+1}^{i}=E^i\bar{x}_{k+1}+\bar{e}_{k+1}^{i}=E^i\big((A-\sum_{j\in\N}B^jL^j)\bar{x}_k-\sum_{j\in\N}B^j\bar{b}_k^{j}\big)+\bar{e}_{k+1}^{i}$. Using this expression in the following, we have   
\begin{align*}
    &Q^i\bar{x}_k+A^{\prime}\bar{\lambda}_{k+1}^{i}-M^{i\prime}\bar{\mu}_k^{i}=Q^i\bar{x}_k+A^{\prime}E^i\big((A-\sum_{j\in\N}B^jL^j)\bar{x}_k\\
    &\quad-\sum_{j\in\N}B^j\bar{b}_k^{j}\big)+A^{\prime}\bar{e}_{k+1}^{i}-M^{i\prime}\bar{\mu}_k^{i}
    \end{align*}
    \begin{align*}
    &=Q^i\bar{x}_k+(A^{\prime}E^iA-A^{\prime}E^{i}B^iL^i\big)\bar{x}_k-\sum_{j\in-i}A^{\prime}E^iB^j\\
    &\quad\times (L^j\bar{x}_k+\bar{b}_k^{j})-A^{\prime}E^{i}B^i\bar{b}_k^i+A^{\prime}\bar{e}_{k+1}^{i}-M^{i\prime}\bar{\mu}_k^{i}\\
    &=\big(Q^i+A^{\prime}E^iA-A^{\prime}E^{i}B^i(Y^i)^{-1}B^{i\prime}E^iA\big)\bar{x}_k\\
    &\quad+A^{\prime}\big(\bar{e}_{k+1}^{i}-\sum_{j\in-i}E^iB^j(L^j\bar{x}_k+\bar{b}_k^{j})\big)-A^{\prime}E^{i}B^i\bar{b}_k^i-M^{i\prime}\bar{\mu}_k^{i}
\end{align*}
where, we have used $L^i=(Y^i)^{-1}B^{i\prime}E^iA$ in the last expression. Next, using \eqref{eq:Eeq} and \eqref{eq:LCS2_E2}, the final expression in the above simplifies to $E^i\bar{x}_k+\bar{e}_k^{i}=\bar{\lambda}_k^{i}$. This implies
\begin{align}
\bar{\lambda}_{k}^{i}=Q^i\bar{x}_k+A^{\prime}\bar{\lambda}_{k+1}^{i}-{M^i}^{\prime}\bar{\mu}_k^{i}, \quad i\in\N. \label{eq:FBeq2}
\end{align}
Further, using \eqref{eq:U5} in \eqref{eq:LCS2_E3}, for each $i\in\N$, we have 
\begin{align}
    &0 \leq M^i\bar{x}_k-\sum_{j\in\N}N^{ij}(R^{jj})^{-1}\big(B^{j\prime}\bar{\lambda}_{k+1}^{j}-{N^{jj}}^{\prime}\bar{\mu}_k^{j}\big)\nonumber\\
    &\hspace{1.5in}+r^i \perp \bar{\mu}_k^{i}\geq 0.\label{eq:CC1}
\end{align}
\end{subequations}
So, (ii) implies solvability of \eqref{eq:LCS_3}. But \eqref{eq:LCS_3} and \eqref{eq:LCS1} represent the same LCS. Thus, (ii) implies the solvability of \eqref{eq:LCS1}.
\end{proof}
\begin{remark}\label{rem:LCS_eqv}
    From the proof of Theorem \ref{th:LCSeqv}, note that, under Assumptions \ref{ass:G1} and \ref{ass:G2}, the equivalence extends beyond solvability of LCSs \eqref{eq:LCS1} and \eqref{eq:LCS2}. Their solutions are related through affine transformations $e_k^{i\star}=\lambda_k^{i\star}-E^ix_k^{\star}$ and $\bar{\lambda}_k^{i}:=E^i\bar{x}_k+\bar{e}_k^{i}$ for $k\in\mathbb{N}_0$ and $i\in\N$.
\end{remark}
\begin{corollary}\label{cor:Th2C1}
    Let Assumptions \ref{ass:G1} and \ref{ass:G2} hold and $x_0\in\X_0$. Then LCS \eqref{eq:LCS1} admits square-summable solutions if and only if LCS \eqref{eq:LCS2} does.
\end{corollary}
\begin{proof}
    The equivalence of solvability of LCSs \eqref{eq:LCS1} and \eqref{eq:LCS2} is established in Theorem \ref{th:LCSeqv}, and by Remark \ref{rem:LCS_eqv}, their solutions are related through affine transformations $e_k^{i\star}=\lambda_k^{i\star}-E^ix_k^{\star}$ and $\bar{\lambda}_k^{i}=E^i\bar{x}_k+\bar{e}_k^{i}$, $\forall k\in\mathbb{N}_0$ and $i\in\N$. So, the proof follows, as square-summability is preserved under these affine transformations.
\end{proof}
The next theorem, building on the previous results, establishes a sufficient condition under which a solution of \eqref{eq:LCS1} yields an OL-GNE.
\begin{theorem}\label{cor:C2}
    Let Assumptions \ref{ass:G1} and \ref{ass:G2} hold. For a given  $x_0\in\X_0$, suppose LCS \eqref{eq:LCS1} is solvable and admits  square-summable solutions, i.e., $(x_k^{\star})_{k=0}^{\infty}\in \ell^2(\R^n)$, $(\lambda_k^{\star})_{k=0}^{\infty}\in \ell^2(\R^{Nn})$, $(\mu_k^{\star})_{k=0}^{\infty}\in \ell^2(\R_{+}^{c})$, where $\lambda_k^{\star}=\col{\lambda_k^{i\star}}_{i=1}^N$ and $\mu_k^{\star}=\col{\mu_k^{i\star}}_{i=1}^N$. 
    Then, for any such square-summable solution, the controls $\bu^{i\star}:=(u_k^{i\star})_{k=0}^{\infty},~i\in\N$ defined by \eqref{eq:BarU1} constitute an OL-GNE of IDGC \eqref{eq:Game1}.
\end{theorem}
\begin{proof}
   From Theorem \ref{cor:C1}, we know, if \eqref{eq:LCS2} is solvable and $(\bar{x}_k)_{k=0}^{\infty}\in \ell^2(\R^n)$, $(\bar{e}_k)_{k=0}^{\infty}\in \ell^2(\R^{Nn})$, $(\bar{\mu}_k)_{k=0}^{\infty}\in \ell^2(\R^{c}_{+})$, then, $\bar{u}_k^{i},~i\in\N$, defined in \eqref{eq:U2}, is an OL-GNE for IDGC \eqref{eq:Game1}. In other words, \eqref{eq:GNEeqTH2} hold for each $i\in\N$. 
   But by the equivalence in Theorem \ref{th:LCSeqv}, Corollary \ref{cor:Th2C1} and in particular from \eqref{eq:Bu2Bs} and \eqref{eq:U5}, we note that for each $i\in\N$, $\bar{u}_k^{i}$ and $u_k^{i\star}$ are exactly same. Therefore, from \eqref{eq:GNEeqTH2}, for all $i\in\N$,
   \begin{align*}
        J^i(x_0, (\bu^{i\star},\bu^{-i\star}))\leq  J^i(x_0, (\bu^i,\bu^{-i\star})), ~\forall \bu^i\in 	\U^i(\bu^{-i\star}; x_0).
   \end{align*}
   So, by Definition \ref{def:OCNEdef}, $(\bu^{i\star},\bu^{-i\star})$ is an OL-GNE for \eqref{eq:Game1}.
\end{proof}
%
%
%
\begin{remark}
     In the unconstrained setting, infinite-horizon OL Nash equilibria for difference games were studied in \cite{Monti:2024}. The proof of \cite[Theorem~4.8]{Monti:2024} invokes \cite[Lemma~4.4]{Monti:2024} by replacing $w(\cdot)$ with the opponent's OL Nash strategies. As \cite[Lemma~4.4]{Monti:2024} relies on the $\ell^2$ property of $w(\cdot)$, its application appears to require the opponent's strategy to belong to $\ell^2$, a property not established at that stage of the proof. Consequently, additional justification is needed for the player-wise application of \cite[Lemma~4.4]{Monti:2024}. Instead, a consistent derivation requires the square-summability of the state-costate sequences, established later under \cite[Assumption~4.9]{Monti:2024}. This note extends \cite{Monti:2024} to the coupled constrained case using different analysis and proof techniques that avoid this issue (Theorems \ref{cor:C1} and \ref{cor:C2}) and, unlike \cite{Monti:2024}, does not require the system matrix $A$ to be invertible.
\end{remark}
%
\section{Reformulation of Sufficient Conditions}\label{sec:Solvability}
 This section reformulates sufficient conditions of Theorem \ref{cor:C2} into more tractable forms. In particular, we make some additional assumptions, under which, we can guarantee that LCS \eqref{eq:LCS1} admits square-summable solutions, i.e., $(x_k^{\star})_{k=0}^{\infty}\in \ell^2(\R^n)$, $(\lambda_k^{\star})_{k=0}^{\infty}\in \ell^2(\R^{Nn})$ and $(\mu_k^{\star})_{k=0}^{\infty}\in \ell^2(\R_{+}^{c})$. This reformulation also leads to a computational procedure for obtaining OL-GNE of IDGC \eqref{eq:Game1}.

 We first present an unconstrained case result in Lemma \ref{lem:L5}, which will be useful for subsequent analysis. To this end, in the absence of inequality constraints, consider the state-costate equations \eqref{eq:LCS1_eq1}-\eqref{eq:LCS1_eq2} of LCS \eqref{eq:LCS1} $\forall k\geq K$ (for some finite $K\in\mathbb{N}_0$) given as follows:
\begin{subequations}\label{eq:KFBeq}
    \begin{align}
        x_{k+1}^{\star}&=Ax_k^{\star}-\sum_{j\in\N}B^j(R^{jj})^{-1}B^{j\prime}\lambda_{k+1}^{j\star}, \label{eq:KFBeq1}\\
        \lambda_{k}^{i\star}&=Q^ix_k^{\star}+A^{\prime}\lambda_{k+1}^{i\star}, \quad i\in\N, \label{eq:KFBeq2}
    \end{align}
\end{subequations}
with initial condition $x_K^{\star}$. For all $i\in\N$, we write \eqref{eq:KFBeq} compactly as
\begin{align}
  \mathbf{M}\begin{bmatrix}
    x_k^{\star}\\
    \lambda_k^{\star}
   \end{bmatrix}=\mathbf{L}\begin{bmatrix}
                   x_{k+1}^{\star}\\
                   \lambda_{k+1}^{\star}
                   \end{bmatrix},\label{eq:SCeq}
\end{align}
where, $\mathbf{M}=
\left[
\begin{smallmatrix}
~A ~~& 0\\[0.3ex]
-Q ~~& I_{Nn}
\end{smallmatrix}
\right]$, $\mathbf{L}=
\left[
\begin{smallmatrix}
I_n&S\\
0&I_N\otimes A^{\prime}
\end{smallmatrix}
\right]$, $Q=\col{Q^i}_{i=1}^N$, $S=\bar{B}R^{-1}B^{\prime}$, $R=\oplus_{i=1}^{N}R^{i i}$, $B=\oplus_{i=1}^{N}B^i$,  $\bar{B}=[B^{1},\cdots, B^{N}]$.
\begin{assumption}\label{ass:A_4}
    The matrix pencil $\mathbf{M}-\lambda \mathbf{L}$ is regular with exactly $n$ generalized eigenvalues in the open unit disk. Moreover, its stable deflating subspace admits a basis of the form $Z=\col{X,Y}$, $X\in\mathbb R^{n\times n},\; Y\in\mathbb R^{Nn\times n}$, with $X$ nonsingular.
\end{assumption}
The stable deflating subspace associated with the pencil $\mathbf{M}-\lambda \mathbf{L}$ is the subspace spanned by all generalized eigenvectors corresponding to generalized eigenvalues inside the unit disk; see \cite{Pappas:1980} for details.
\begin{lemma}\label{lem:L5}
     Let Assumption \ref{ass:A_4} hold. Then every state-costate trajectory \eqref{eq:SCeq} evolving in the stable deflating subspace of the pencil $\mathbf{M}-\lambda \mathbf{L}$ satisfies $\lambda_k^{\star}=Px_k^{\star},~\forall k\geq K$, with $P=YX^{-1}$. Moreover, if $I+SP$ is invertible, then $P$ satisfies the following coupled Riccati equations
     \begin{align}
         P=Q+(I_N\otimes A^{\prime})P(I+SP)^{-1}A,\label{eq:RiccatiP}
     \end{align}
     and the matrix $(I+SP)^{-1}A$ is Schur stable.
\end{lemma}
\begin{proof}
For any state-costate trajectory evolving in the stable deflating subspace of the pencil $\mathbf{M}-\lambda \mathbf{L}$ (with basis $Z=\col{X,Y}$),  there exists a vector $\eta_k$ such that $\begin{bmatrix}
    x_k^{\star}\\
    \lambda_k^{\star}
   \end{bmatrix}=\begin{bmatrix}
       X\\Y
   \end{bmatrix}\eta_k,~k\geq K$. As $X$ is nonsingular, we have $\eta_k=X^{-1}x_k^{\star}$ and $\lambda_k^{\star}=Y\eta_k=YX^{-1}x_k^{\star}=Px_k^{\star}$ for all $k\geq K$.
Also, as $Z$ spans the stable deflating subspace of the pencil $\mathbf{M}-\lambda \mathbf{L}$, there exists a matrix $G_s\in\mathbb R^{n\times n}$ with eigenvalues in the open unit disk such that $\mathbf{M}Z=\mathbf{L}ZG_s$; see \cite{Pappas:1980}. With $P=YX^{-1}$, this relation can be written as
\begin{align}
    \mathbf{M}\begin{bmatrix}
                I\\
                P
\end{bmatrix}X=\mathbf{L}\begin{bmatrix}
                          I\\
                          P
                        \end{bmatrix}XG_s \implies \mathbf{M}\begin{bmatrix}
                                                               I\\
                                                               P
                                                              \end{bmatrix}
=\mathbf{L}\begin{bmatrix}
I\\
P
\end{bmatrix}G^1, \label{eq:SDS}
\end{align}
where the second expression is obtained by post multiplying $X^{-1}$ and defining $G^1=XG_sX^{-1}$. Since, $G^1$ is similar to $G_s$, all eigenvalues of $G^1$ are also inside the open unit disk i.e. $G^1$ is Schur stable.
Expanding both sides of \eqref{eq:SDS}, gives $\begin{bmatrix}
A\\
P-Q
\end{bmatrix}
=
\begin{bmatrix}
I+SP\\
(I_N\otimes A^{\prime})P
\end{bmatrix}G^1$.
Equating the first block rows yields $A=(I+SP)G^1$ or $G^1=(I+SP)^{-1}A$, thus, $(I+SP)^{-1}A$ is Schur stable. Similarly, equating the second block rows yields coupled Riccati equations \eqref{eq:RiccatiP}. 
\end{proof}
\begin{remark}\label{rem:Px}
    Since solutions of \eqref{eq:KFBeq} converging to the origin must lie in the stable deflating subspace of the pencil $\mathbf{M}-\lambda \mathbf{L}$, it follows from Lemma \ref{lem:L5} that, under Assumption \ref{ass:A_4}, every such stabilizing solution satisfies $\lambda_k^{\star}=Px_k^{\star}$ with $x_{k+1}^{\star}=G^1x_k^{\star}$ (from \eqref{eq:KFBeq1}), for all $k\geq K$.
\end{remark}
To compactly express subsequent expressions, we define some notations: $G^1=(I+SP)^{-1}A$, $G^2=-(I+SP)^{-1}S$, $G^3=(I+SP)^{-1}\bar{B}R^{-1}(\otimes_{i=1}^{N}{N^{ii}})^{\prime}$, $H^1=(I_N\otimes A^{\prime})(I-PG^2)$, $H^2=(I_N\otimes A^{\prime})PG^3-(\otimes_{i=1}^{N}{M^i})^{\prime}$, $F^{1}=-R^{-1}B^{\prime}PG^1$, $F^{2}=-R^{-1}B^{\prime}(I+PG^2)$, $F^{3}=R^{-1}((\otimes_{i=1}^{N}{N^{ii}})^{\prime}-B^{\prime}PG^3)$, $\bu_k^{\star}=\col{u_k^{i\star}}_{i=1}^{N}$ and $\mu_k^{\star}
=\col{\mu_k^{i\star}}_{i=1}^N$.
In addition, we recall some early notations $\bar{M}=\col{M^i}_{i=1}^{N}$, $\bar{N}=\col{[N^{i1} ~N^{i2} \cdots N^{iN}]}_{i=1}^{N}$,  $r=\col{r^i}_{i=1}^N$ defined below \eqref{eq:vectorC1}. Using these notations, the complementarity conditions in \eqref{eq:LCS1_eq3}, $\forall i\in\N$, can be compactly expressed as $0 \leq \bar{M}x_k^{\star}-\bar{N}R^{-1}B^{\prime}\lambda_{k+1}^{\star}+\bar{N}R^{-1}(\otimes_{i=1}^{N}{N^{ii}})^{\prime}\mu_k^{\star}+r \perp \mu_k^{\star}\geq 0$. By adding and subtracting  the term $\bar{N}R^{-1}B^{\prime}PG^1x_k^{\star}=-\bar{N}F^1x_k^{\star}$ to the first inequality, it can be expressed as follows:
\begin{align}
    &0 \leq (\bar{M}+\bar{N}F^1)x_k^{\star}-\bar{N}R^{-1}B^{\prime}(\lambda_{k+1}^{\star}-PG^1x_k^{\star})\nonumber\\
    &\hspace{0.25in}+\bar{N}R^{-1}(\otimes_{i=1}^{N}{N^{ii}})^{\prime}\mu_k^{\star}+r \perp \mu_k^{\star}\geq 0,~\forall k\in\mathbb{N}_0.\label{eq:LCS1_eq3R}
\end{align}
Further, let LCS$|_{K}$ represent the restriction of the infinite-horizon LCS \eqref{eq:LCS1} to a finite horizon $K$, for some given $K\in \mathbb{N}_0$.

The next result presents a reformulation of sufficient conditions. 
\begin{theorem}\label{th:Th_K}
     Let Assumptions \ref{ass:G1}, \ref{ass:G2} and \ref{ass:A_4} hold. Let for a given $x_0\in\X_0$, there exist a finite $K\in\mathbb{N}_0$ such that:
     \begin{enumerate}[label=(\roman*)]
        \item \label{item:Th4_i} LCS$|_{K}$ is solvable with terminal conditions $\lambda_{K}^{\star}=Px_K^{\star}$, $\mu_K^{\star}=0$, where $P$ is the solution of coupled Riccati equations \eqref{eq:RiccatiP} and 
        \item \label{item:Th4_ii} $(\bar{M}+\bar{N}F^1)\phi(k-K)x_K^{\star}+r>0$ for all $k\geq K$, where, $\phi(\tau)=(G^1)^{\tau}$ for $\tau>0$ and $\phi(0)=I$ with $G^1=(I+SP)^{-1}A$.
     \end{enumerate} 
    Then, LCS \eqref{eq:LCS1} admits square-summable solutions and the corresponding controls given by \eqref{eq:BarU1} constitute an OL-GNE for IDGC \eqref{eq:Game1}.
\end{theorem}
\begin{proof}
    For any stabilizing solutions of \eqref{eq:KFBeq}, by Remark \ref{rem:Px} we have, $\lambda_{k+1}^{\star}=Px_{k+1}^{\star}=PG^1x_k^{\star}$ and $x_k^{\star}=\phi(k-K)x_K^{\star},~\forall k\geq K$. Hence, by condition \ref{item:Th4_ii} of the theorem, it is straightforward to verify that, for such stabilizing solution of \eqref{eq:KFBeq}, the complementarity condition \eqref{eq:LCS1_eq3R} holds $\forall k\geq K$ with $\mu_k^{\star}=0,~k\geq K$ (It is worth noting that, even under $\lambda_{k+1}^{\star}=PG^1x_k^{\star}$ and condition \ref{item:Th4_ii}, \eqref{eq:LCS1_eq3R} may admit nonzero solutions $\mu_k^{\star}\geq 0$ for $k\geq K$. Thus, $\mu_k^{\star}=0,~\forall k\geq K$ is merely one admissible choice). Comparing \eqref{eq:KFBeq} and \eqref{eq:LCS1_eq3R} with LCS \eqref{eq:LCS1}, we note that, for all $k\geq K$, $\lambda_k^{\star}=Px_k^{\star}$ and $\mu_k^{\star} = 0$ solve the LCS \eqref{eq:LCS1}. Therefore, any solution of LCS$|_{K}$  with terminal conditions $\lambda_{K}^{\star}=Px_K^{\star}$, $\mu_K^{\star}=0$ (which exist by condition \ref{item:Th4_i} of the theorem) can be extended to a solution of infinite-horizon LCS \eqref{eq:LCS1} by setting $\lambda_k^{\star}=Px_k^{\star}$ (i.e., by appending stabilizing solution of \eqref{eq:KFBeq}) and $\mu_k^{\star} = 0,~\forall k\geq K$. This extension is well defined due to the consistent boundary conditions at $K$, namely, $\lambda_{K}^{\star}=Px_K^{\star}$, $\mu_K^{\star}=0$.

    Next, we show that the above constructed solution of LCS \eqref{eq:LCS1} is square-summable. As $\mu_k^{\star}=0$ for all $k\geq K$, $\sum_{k=0}^{\infty}\normx{\mu_k^{\star}}^2=\sum_{k=0}^{K-1}\normx{\mu_k^{\star}}^2<\infty$ and thus $(\mu_k^{\star})_{k=0}^{\infty}\in \ell^2(\R^{c}_{+})$. Also,  $x_{k+1}^{\star}=G^1x_k^{\star}$, $\forall k\geq K$ and $G^1$ is Schur stable (Lemma \ref{lem:L5}). So, by Lemma \ref{lem:L00}.(i), $(x_k^{\star})_{k=K}^{\infty}$ is square-summable. Since $\lambda_k^{\star}=Px_k^{\star},~k\geq K$, $(\lambda_k^{\star})_{k=K}^{\infty}$ is also square-summable (i.e., $\sum_{k=K}^{\infty}\normx{x_k^{\star}}^2<\infty$ and $\sum_{k=K}^{\infty}\normx{\lambda_k^{\star}}^2<\infty$). Since $K\in\mathbb{N}_0$ is finite, it follows that $\sum_{k=0}^{\infty}\normx{x_k^{\star}}^2=\sum_{k=0}^{K-1}\normx{x_k^{\star}}^2+\sum_{k=K}^{\infty}\normx{x_k^{\star}}^2<\infty$ and $\sum_{k=0}^{\infty}\normx{\lambda_k^{\star}}^2=\sum_{k=0}^{K-1}\normx{\lambda_k^{\star}}^2+\sum_{k=K}^{\infty}\normx{\lambda_k^{\star}}^2<\infty$.
    As all conditions of Theorem \ref{cor:C2}  hold, the proof follows from Theorem \ref{cor:C2}.
\end{proof}
\begin{remark}\label{rem:Sol_loss}
    As discussed in the proof of Theorem \ref{th:Th_K}, enforcing $\mu_k^{\star}=0$ for all $k\geq K$ can exclude solutions with $\mu_k^{\star}\neq 0$ for $k\geq K$, although this loss diminishes as $K$ increases. Moreover, note that, if the conditions of  Theorem \ref{th:Th_K} hold for some finite $K$, they also hold for any $\bar{K}\geq K$. So, larger choices of $K$ can be used to reduce the potential loss of solutions (also see Remark \ref{rem:Kchoice} and Section \ref{sec:Numerical}).
    %
\end{remark}
\begin{remark}\label{rem:r_implied}
    If condition \ref{item:Th4_ii} in Theorem \ref{th:Th_K} holds $\forall k\geq K$, then, $r> 0$ as $x_k^{\star} \to 0$. Thus, assumption $r> 0$, although not made explicitly, is implied by condition \ref{item:Th4_ii} of Theorem \ref{th:Th_K} (compare with Remark \ref{rem:ri_A1}).
\end{remark}
\begin{remark}\label{rem:LCS2LCP}
    Although Theorem \ref{th:Th_K} provides a constructive proof, two issues remain. First, $K$ is not known \emph{a priori}. Second, verifying condition \ref{item:Th4_i} of Theorem \ref{th:Th_K} is nontrivial.
\end{remark}
In view of Remark \ref{rem:LCS2LCP}, we again reformulate the condition \ref{item:Th4_i} of Theorem \ref{th:Th_K}. To this end, we first present the following Lemma.
\begin{lemma}\label{lem:L6}
     Let Assumptions \ref{ass:G1}, \ref{ass:G2} and \ref{ass:A_4} hold and for a given finite $K\in\mathbb{N}_0$, $\lambda_{K}^{\star}=Px_K^{\star}$, $\mu_k^{\star}=0,~k\geq K$. Then the state-costate trajectories \eqref{eq:LCS1_eq1}-\eqref{eq:LCS1_eq2} of LCS \eqref{eq:LCS1}, satisfy $\lambda_k^{\star}=Px_k^{\star}+\zeta_k,~\forall k< K$, where $P$ is the solution of coupled Riccati equations \eqref{eq:RiccatiP} and $\zeta_k$, with $\zeta_K=0$, satisfy the following backward linear recursive equation 
			\begin{align}
				\zeta_k=H^1\zeta_{k+1}+H^2\mu_k^{\star},~~ \forall k< K.  \label{eq:Zetabackward1}
			\end{align}
\end{lemma}
\begin{proof}
  Using the definitions of $G^1,G^2$ and $G^3$, \eqref{eq:LCS1_eq1}, for all $k<K$ can be written as follows:
    \begin{align}
        x_{k+1}^{\star}&=G^1x_k^{\star}+G^2\big(\lambda_{k+1}^{\star}-(Px_{k+1}^{\star}+\zeta_{k+1})\big)\nonumber\\
        &\quad+G^2\zeta_{k+1}+G^3\mu_k^{\star}.\label{eq:X5}
    \end{align}
    Similarly, $\forall i\in\N$ and $k<K$, \eqref{eq:LCS1_eq2} can be compactly expressed as $\lambda_k^{\star}-(Px_k^{\star}+\zeta_k)=H^1\big(\lambda_{k+1}^{\star}-(Px_{k+1}^{\star}+\zeta_{k+1})\big)+(Q+(I_N\otimes A^{\prime})P(I+SP)^{-1}A-P)x_k^{\star}+(H^1\zeta_{k+1}+H^2\mu_k^{\star}-\zeta_k)=H^1\big(\lambda_{k+1}^{\star}-(Px_{k+1}^{\star}+\zeta_{k+1})\big)$.
    The last expression is due to \eqref{eq:RiccatiP} and \eqref{eq:Zetabackward1}. As $\zeta_K=0$ and $\lambda_{K}^{\star}=Px_K^{\star}$, we have 
     $\lambda_K^{\star}-(Px_K^{\star}+\zeta_K)=0$ and thus from the above recursive relation, $\lambda_k^{\star}=Px_k^{\star}+\zeta_k, \forall k<K$.
\end{proof}
\noindent Next, using the relation $\lambda_k^{\star}=Px_k^{\star}+\zeta_k$, in \eqref{eq:X5}, we have 
\begin{align}
    x_{k+1}^{\star}&=G^1x_k^{\star}+G^2\zeta_{k+1}+G^3\mu_k^{\star}.\label{eq:X6}
\end{align}
Similarly, using $\lambda_{k+1}^{\star}-PG^1x_k^{\star}=\zeta_{k+1}$, in \eqref{eq:BarU1} and \eqref{eq:LCS1_eq3R}, we obtain
    \begin{align}
    &\bu_k^{\star}=F^1x_k^{\star}+F^2\zeta_{k+1}+F^3\mu_k^{\star},\label{eq:M_u}\\
    &0 \leq (\bar{M}+\bar{N}F^1)x_k^{\star}+\bar{N}F^2\zeta_{k+1}+\bar{N}F^3\mu_k^{\star}+r \perp \mu_k^{\star} \geq 0.\label{eq:M_LCP}
\end{align}
Next, similar to \cite{Partha:2026a}, we express \eqref{eq:X6}, \eqref{eq:Zetabackward1} and \eqref{eq:M_LCP} for all $k<K$ as a single linear complementarity problem (LCP). 
From \eqref{eq:Zetabackward1}, with $\zeta_K=0$ (see Lemma \ref{lem:L6}), we have
\begin{subequations}\label{eq:FB_xZeta}
    \begin{align}
    \zeta_k=\sum_{\tau=k}^{K-1}\varphi(\tau-k)H^2\mu_{\tau}^{\star},\label{eq:Zeta1}
\end{align}
where, $\varphi(\tau)=(H^1)^{\tau}$ for $\tau>0$ and $\varphi(0)=I$. Similarly from \eqref{eq:X6}, we have
\begin{align*}
    x_k^{\star}=\phi(k)x_{0}+\sum_{\tau=0}^{k-1}\phi(k-\tau-1)G^2\zeta_{\tau+1}+\sum_{\tau=0}^{k-1}\phi(k-\tau-1)G^3\mu_{\tau}^{\star}
\end{align*}
where, $\phi(\tau)=(G^1)^{\tau}$ for $\tau>0$ and $\phi(0)=I$. Using \eqref{eq:Zeta1}, we can eliminating $\zeta_{k}$, from the above state equation as follows:
\begin{align}
    x_k^{\star}&=\phi(k)x_0+ \sum_{\rho=1}^{K-1}
    \Big(\sum_{\tau=0}^{\min(k-1,\rho-1)}\phi(k-\tau-1)G^2\varphi(\rho-\tau-1)\Big)\nonumber\\
    &\hspace{0.5in}\times H^2 \mu_\rho^\star+ \sum_{\tau=0}^{k-1}\phi(k-\tau-1)G^3\mu_\tau^{\star}
\end{align}
\end{subequations}

Finally, we stack all variables for $k<K$ as $x_{\K}^{\star}=\col{x_{k}^{\star}}_{k=0}^{K-1}$, $\bu^{\star}_\K=\col{u^{\star}_k}_{k=0}^{K-1}$, $\upmu^{\star}_\K=\col{\mu^{\star}_k}_{k=0}^{K-1}$ and $\zeta_\K=\col{\zeta_{k+1}}_{k=0}^{K-1}$, to write the forward and backward equations \eqref{eq:FB_xZeta} compactly as:
\begin{subequations}\label{eq:X_Zeta1}
    \begin{align}
    &x_{\K}^{\star}=\mathbf{\Phi}_0x_0+\mathbf{\Phi}_1\upmu_{\K}^{\star},\label{eq:X_K}\\
    &\zeta_\K=\mathbf{\Psi}_1\upmu_{\K}^{\star},
    \end{align}
\end{subequations}
where $\mathbf{\Phi}_0=\col{\phi(k)}_{k=0}^{K-1}$, $[\mathbf{\Phi}_1]_{k\tau}=0$ for $k=1$, $[\mathbf{\Phi}_1]_{k\tau}=\phi(k-\tau-1)G^3+\sum_{\rho=0}^{\tau-2}
\phi(k-\rho-2)\,G^2\,\varphi(\tau-\rho-2)\,H^2$ for $k>\tau \geq 1$, $[\mathbf{\Phi}_1]_{k\tau}=\sum_{\rho=0}^{k-2}
\phi(k-\rho-2)\,G^2\,\varphi(\tau-\rho-2)\,H^2$ for $2\leq k\leq \tau\leq K$, $[\mathbf{\Psi}_1]_{k\tau}=\varphi(\tau-k-1)H^2$ for $\tau>k$ and  $[\mathbf{\Psi}_1]_{k\tau}=0$ for $\tau\leq k$ with $k,\tau=1,\dots,K$.
Similarly, \eqref{eq:M_u} and \eqref{eq:M_LCP}, $\forall k<K$, can be written as
 \begin{align}
    &\bu^{\star}_\K=(I_K\otimes F^1)x_{\K}^{\star}+(I_K\otimes F^2)\zeta_\K+(I_K\otimes F^3)\upmu^{\star}_\K, \label{eq:U_K1}\\
    & 0 \leq (I_K\otimes(\bar{M}+\bar{N}F^1))x_{\K}^{\star}+(I_K\otimes\bar{N}F^2)\zeta_\K\nonumber\\
    &\hspace{0.35in}+(I_K\otimes\bar{N}F^3)\upmu^{\star}_\K+(1_K\otimes r) \perp \upmu^{\star}_\K \geq 0.\label{eq:LCP_K}
\end{align}

The next result presents an LCP reformulation. To this end, we introduce some additional notation. Let $\mathsf{M}=(I_K\otimes(\bar{M}+\bar{N}F^1))\mathbf{\Phi}_1+(I_K\otimes\bar{N}F^2)\mathbf{\Psi}_1+(I_K\otimes\bar{N}F^3)$, $\mathsf{q}=(I_K\otimes(\bar{M}+\bar{N}F^1))\mathbf{\Phi}_0,~\mathsf{L}=(I_K\otimes F^1)\mathbf{\Phi}_0$, $\mathsf{F}=(I_K\otimes F^1)\mathbf{\Phi}_1+(I_K\otimes F^2)\mathbf{\Psi}_1+(I_K\otimes F^3)$ and $\bar{G}_k=(G^1)^{k-K}(G^1[\mathbf{\Phi}_1]_{K}+[0_{n\times m(K-1)} ~G^3])$.
\begin{theorem}\label{th:LCP}
    Let Assumptions \ref{ass:G1}, \ref{ass:G2} and \ref{ass:A_4} hold. Let for a given $x_0\in\X_0$, there exist a finite $K\in\mathbb{N}_0$ such that  the following  large-scale linear complementarity problem  
	\begin{align}
		\mathrm{LCP}(x_0):\quad 0 \leq \mathsf{M}\upmu_{\K}^{\star}+\mathsf{q}x_0+(1_K\otimes r) \perp \upmu_{\K}^{\star} \geq 0, \label{eq:FinalOLNElcp}
	\end{align}
    is solvable and $\big(\bar{M}+\bar{N}F^1\big)\big((G^1)^{k}x_0+\bar{G}_k\upmu_{\K}^{\star}\big)+r>0$ for all $k\geq K$, where $P$ is the solution of the coupled Riccati equations \eqref{eq:RiccatiP}. 
    Then, OL-GNE strategy profile of \eqref{eq:Game1}, $\forall k<K$ is given by 
	\begin{align}
		\bu_\K^{\star}=\mathsf{L}x_0+\mathsf{F}\upmu_{\K}^{\star}, \label{eq:FinalOLNE}
	\end{align}
	and for $k\geq K$ is given by 
    \begin{align}
        \bu_{k}^{\star}&=F^1\big((G^1)^{k}x_0+\bar{G}_k\upmu_{\K}^{\star}\big).\label{eq:U_afterK}
    \end{align}
\end{theorem}
\begin{proof}
The LCP \eqref{eq:FinalOLNElcp} is obtained by substituting \eqref{eq:X_Zeta1} in \eqref{eq:LCP_K}. So, by Lemma \ref{lem:L6} and the steps before this theorem, we note that solvability of LCP \eqref{eq:FinalOLNElcp} is same as solvability of LCS$|_{K}$ with terminal conditions $\lambda_{K}^{\star}=Px_K^{\star}$, $\mu_K^{\star}=0$. Further, if LCP \eqref{eq:FinalOLNElcp} is solvable, then from \eqref{eq:X_K}, we have $x_{K-1}^{\star}=[\mathbf{\Phi}_0]_{K}x_0+[\mathbf{\Phi}_1]_{K}\upmu_{\K}^{\star}$. So, from \eqref{eq:X6}, the equilibrium state at $K$ (with $\zeta_K=0$ from Lemma \ref{lem:L6}) is given by $x_K^{\star}=G^1 x_{K-1}^{\star}+G^3\mu_{K-1}^{\star}=G^1[\mathbf{\Phi}_0]_{K}x_0+G^1[\mathbf{\Phi}_1]_{K}\upmu_{\K}^{\star}+G^3\mu_{K-1}^{\star}$. Also, using $[\mathbf{\Phi}_0]_{K}=\phi(K-1)=(G^1)^{K-1}$, we can express $x_K^{\star}$ as
\begin{align}
    x_K^{\star}=(G^1)^{K}x_0+(G^1[\mathbf{\Phi}_1]_{K}+[0_{n\times m(K-1)}~G^3])\upmu_{\K}^{\star}.\label{eq:Th5_ii}
\end{align}
Thus, by \eqref{eq:Th5_ii}, the condition  $\big(\bar{M}+\bar{N}F^1\big)\big((G^1)^{k}x_0+\bar{G}_k\upmu_{\K}^{\star}\big)+r>0,~\forall k\geq K$, is equivalent to the condition \ref{item:Th4_ii} in Theorem \ref{th:Th_K}. Therefore, under the assumptions of this theorem, all conditions of Theorem \ref{th:Th_K} hold. Moreover, in view of Lemma \ref{lem:L6}, \eqref{eq:U_K1} is equivalent to \eqref{eq:BarU1} for all $k<K$ and \eqref{eq:FinalOLNE} is obtained by substituting \eqref{eq:X_Zeta1} in \eqref{eq:U_K1}. Thus, OL-GNE strategy profile of \eqref{eq:Game1} for all $k<K$ is given by \eqref{eq:FinalOLNE}. Finally, using $\mu_k^{\star}=0$, $\lambda_{k+1}^{\star}=Px_{k+1}^{\star}=P(G^1)^{k+1-K}x_K^{\star},~\forall k\geq K$ and \eqref{eq:Th5_ii} in \eqref{eq:BarU1},  OL-GNE strategy profile of all players $\forall k\geq K$ is given by \eqref{eq:U_afterK}.
\end{proof}
\begin{remark}\label{rem:LCP_dim}
    OL-GNE for all $k<K$ can be obtained by directly solving the mixed LCP formed by \eqref{eq:X6}, \eqref{eq:Zetabackward1}, and \eqref{eq:M_LCP}, in $2N(K-1)n+NKc$ number of decision variables. In contrast, by Theorem \ref{th:LCP}, the same can be obtained by solving the large scale LCP \eqref{eq:FinalOLNElcp} in only $NKc$ variables. This reformulation is also advantageous as existence results and numerical solution methods for LCPs are well established in the optimization community; see \cite{Pang:2009}.
    %
\end{remark}
\begin{remark}\label{rem:Kchoice}
    A limitation of Theorem \ref{th:LCP} is that $K$ is not known \emph{a priori}. In practice, the horizon $K$ can gradually increased until  LCP \eqref{eq:FinalOLNElcp} is solvable and $\big(\bar{M}+\bar{N}F^1\big)\big((G^1)^{k}x_0+\bar{G}_k\upmu_{\K}^{\star}\big)+r>0$ hold $\forall k\geq K$. 
    While larger $K$ improves solution quality (Remark \ref{rem:Sol_loss}), it also increases the dimension of \eqref{eq:FinalOLNElcp} and thus the computational cost (Remark \ref{rem:LCP_dim}).
    For alternative numerical approach, see the receding horizon approach presented in some recent works \cite{Benenati:2026}, \cite{Baghbadorani:2026}.
\end{remark}
\begin{remark}
    As $G^1$ is Schur stable (Lemma \ref{lem:L5}), the term $r>0$ (see Remark \ref{rem:r_implied}) dominates
    in $(\bar{M}+\bar{N}F^1)((G^1)^{k}x_0+\bar{G}_k\upmu_{\K}^{\star})+r$ for large $k$. Thus, in practice, the condition $(\bar{M}+\bar{N}F^1)((G^1)^{k}x_0+\bar{G}_k\upmu_{\K}^{\star})+r>0$ needs to be checked only for finitely many $k\geq K$.
\end{remark}
\section{Numerical Illustration}\label{sec:Numerical}
We consider the discrete-time vehicle platooning example from \cite{Benenati:2026, Baghbadorani:2026} with sampling period $\tau_s=0.1s$. The platoon consists of $N$ vehicles. The leader ($i=1$) tracks a reference velocity $v^{ref}$, while each follower $i\in\N\setminus\{1\}$ regulates its velocity $v_k^i$ to match that of the predecessor and maintains a desired spacing $d^i+h^iv_k^i$, where $h^i$ denotes the time-headway parameter. 
For all vehicles except leader, local state $x_k^i$ is defined in terms of corresponding tracking and spacing errors as $x^i_k=\left[
\begin{smallmatrix}
        p^{i-1}_k-p^{i}_k-d^i-h^iv^i_k\\
        v_k^{i-1}-v_k^i
    \end{smallmatrix}
\right], ~\forall i\in\N\setminus\{1\}$.
Since the leader’s relative position to itself is zero, her state is defined as $x^1_k=\col{0, v^{ref}-v^1_k}$. 
Stacking all agent states as $x_k=\col{x_k^i}_{i=1}^{N}$ with inputs $u_k^i$, the overall dynamics can be written as $x_{k+1}=Ax_k+\sum_{i\in\N}B^iu_k^i$, where $A=\oplus\left(\left[
\begin{smallmatrix}
        0 &0\\
        0 &1
    \end{smallmatrix}
\right], I_{N-1}\otimes\left[
\begin{smallmatrix}
        1 &\tau_{s}\\0 &1
    \end{smallmatrix}
\right]\right)$, $B^i=\delta^{i+1}_N \otimes \left[
\begin{smallmatrix}
        \tau^2_s/2\\\tau_s
    \end{smallmatrix}
\right]-\delta^i_N \otimes \left[
\begin{smallmatrix}
        h^i\tau_s+\tau^2_s/2\\\tau_s
    \end{smallmatrix}
\right],~i\in\N\setminus\{1,N\}$, $B^1=\delta^2_N \otimes \left[
\begin{smallmatrix}
        \tau^2_s/2\\\tau_s
    \end{smallmatrix}
\right]-\delta^1_N \otimes \left[
\begin{smallmatrix}
        0\\\tau_s
    \end{smallmatrix}
\right]$ and $B^N=-\delta_N^N \otimes \left[
\begin{smallmatrix}
        h^N\tau_s+\tau^2_s/2\\\tau_s
    \end{smallmatrix}
\right]$
with $\delta_N^{i}\in\R^N$ denoting the $i$th canonical basis vector.
Since the system violates stabilizability Assumption \ref{ass:G1}.(iii), following \cite{Benenati:2026, Baghbadorani:2026}, we employ the local prestabilizing controller $K^i_{stab}=(\delta_N^i)^{\prime}\otimes[-1~~ -1]$ to each agent $i\in\N$. Note that the system matrix is singular since its first row is zero.
Further, $\forall k\in\mathbb{N}_0$, following constraints are imposed:
\begin{subequations}\label{eq:VPconstraints}
\begin{align}
    &p_{k+1}^{i-1}\geq d^{min}+p_{k+1}^i,~\quad\text{(collision avoidance)},\label{eq:VP_C1}\\
    &v^{min}\leq v_{k+1}^i\leq v^{max},\quad \text{(velocity limits)},\label{eq:VP_C2}\\
    &u^{min}\leq u_k^i\leq u^{max},~~~\quad \text{(control limits)}.\label{eq:VP_C3}
\end{align}
\end{subequations}
We define $u^i_k=-K^i_{stab}x_k+\bar{u}_k^i$ and $\bar{A}=A-\sum_{i\in\N}B^iK^i_{stab}$. Also, note that from the state definitions, we have $v_k^i=v^{ref}-\sum_{j=1}^{i}[x_k^j]_2,~i\in\N$ and $p_{k+1}^{i-1}-p_{k+1}^i=[x_{k+1}^i]_1+h^iv_{k+1}^i+d^i, i\in\N\setminus\{1\}$. Thus, \eqref{eq:VP_C1} for all $i\in\N\setminus\{1\}$ and \eqref{eq:VP_C2}-\eqref{eq:VP_C3} for all $i\in\N$ can be written as
\begin{align*}
    \big((\delta_{2N}^1)^{\prime}+h^iC^i\big)x_{k+1}+(h^iv^{ref}+d^i-d^{min})&\geq 0,\\
    C^ix_{k+1}+v^{ref}-v^{min}&\geq 0,\\
    -C^ix_{k+1}+v^{max}-v^{ref}&\geq 0,\\
    -K^i_{stab}x_k+\bar{u}_k^i-u^{min}&\geq 0,\\
    K^i_{stab}x_k-\bar{u}_k^i+u^{max}&\geq 0,
\end{align*}
where, $C^i=(\sum_{j=1}^{i}\delta_N^j)^{\prime}\otimes [0~~-1],~i\in\N$.
In view of Remark \ref{rem:purestateC}, we reformulate the above pure state constraints and express this vehicle platooning example in standard form \eqref{eq:Game1} using the state dynamics $x_{k+1}=\bar{A}x_k+\sum_{i\in\N}B^i\bar{u}_k^i$ (including prestabilizing local controller). All inequality constraints can be expressed in the form \eqref{eq:Gconstraints}, with
\begin{align*}
    M^i=\begin{bmatrix}
        \col{((\delta_{2N}^1)^{\prime}+h^jC^j)\bar{A}}_{j=2}^{N}\\
        ~\col{C^j\bar{A}}_{j=1}^{N}\\
        -\col{C^j\bar{A}}_{j=1}^{N}\\
        -K^i_{stab}\\
        ~K^i_{stab}
    \end{bmatrix}, 
\end{align*}
\begin{align*}
    N^{ii}=\begin{bmatrix}
        \hat{N}\\
        ~\col{C^jB^i}_{j=2}^{N}\\
        -\col{C^jB^i}_{j=2}^{N}\\
         ~1\\
         -1
    \end{bmatrix},~N^{ij}=\begin{bmatrix}
        \hat{N}\\
        ~\col{C^lB^j}_{l=2}^{N}\\
        -\col{C^lB^j}_{l=2}^{N}\\
         0\\
         0
    \end{bmatrix},
\end{align*}
\begin{align*}
   r^i=\begin{bmatrix}
        \col{(h^jv^{ref}+d^j-d^{min}}_{j=2}^{N}\\
        (v^{ref}-v^{min})1_{5}\\
        (v^{max}-v^{ref})1_{5}\\
        -u^{min}\\
        ~u^{max}
    \end{bmatrix}
\end{align*}
where, $\hat{N}:=\col{((\delta_{2N}^1)^{\prime}+h^lC^l)B^j}_{l=2}^{N}$.  
For numerical illustration, we take $N=5$, $v^{ref}=30$, $v^{max}=35$, $v^{min}=25$, $u^{max}=2$, $u^{nin}=-2$, $d^{min}=10$, $Q^i=I$, $R^{ii}=1$, $R^{ij}=0,~i,j\in\N,i\neq j$,  $d^i=15,~i\in\N\setminus\{1\}, ~x_0=\col{0,-3, -10,3.5, -11,2, -4.5,-3, -10,4}$.
For these numerical values, Assumptions \ref{ass:G1}, \ref{ass:G2} and \ref{ass:A_4} are satisfied. 

Fig. \ref{fig:Fig1} shows vehicle positions (relative to the leading agent) and velocities for $K=55$ (the conditions of Theorem \ref{th:LCP} hold for any $K\ge 45$). In Fig \ref{fig:VPfig1}, the dotted lines with same colors represent the desired steady state values and shaded regions represent distance constraints. In Fig \ref{fig:VPfig2}, the dotted line denotes $v^{ref}$ and the dashed black lines represent the constraints \eqref{eq:VP_C2}. All vehicle velocities converge to the reference value without violating any of the constraints in \eqref{eq:VPconstraints}.
Fig. \ref{fig:Fig2} shows two OL-GNE (solid and dashed). Since the conditions of Theorem \ref{th:LCP} hold for both $K=45$ and $K=55$, OL-GNE can be computed in either case. Both equilibria in Fig. \ref{fig:Fig2} can be obtained for $K=55$, since $\mu_{k}^{\star}=0$ for all $k\geq 47$ in both cases. In contrast, with $K=45$, the proposed reformulation recovers only equilibria satisfying $\mu_{k}^{\star}=0,~\forall k\geq 45$. 
Accordingly, although the solid-line trajectory is an OL-GNE of the original game, for  $K=45$, it is not recovered by the reformulation as $\mu_{45}^{\star}$ and  $\mu_{46}^{\star}$ are non-zero. The dashed-line equilibrium, however, satisfies $\mu_{k}^{\star}=0,~\forall k\geq 45$ and thus can be obtained with $K=45$. 
This illustrates that the choice of $K$ influences the set of OL-GNE that can be obtained through the proposed reformulation; see Remarks \ref{rem:Sol_loss} and \ref{rem:Kchoice}.

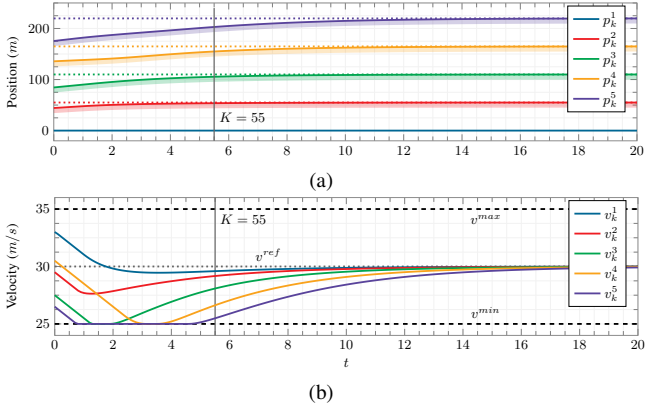
\begin{figure}[t] 
	\centering 
	\input{VP_2}
	\caption{Positions with respect to leading agent (panel (a)) and velocities (panel (b)) of all agents over time $(t=\tau_s k)$, with $K=55$. 
    }  
	\label{fig:Fig1}
\end{figure}
\section{Conclusion}\label{sec:Conclusions}
In this note, we studied a class of infinite-horizon LQ difference games with coupled affine inequality constraints. We derived necessary conditions for the existence of open-loop generalized Nash equilibria (OL-GNE) and established their sufficiency under additional assumptions. The sufficient conditions are expressed in terms of square-summable solutions of associated infinite-horizon linear complementarity systems. Further, we reformulated the sufficient conditions and showed that the computation of OL-GNE can be reduced to solving a large-scale LCP together with verifying additional conditions. Future work will investigate the existence and computation of feedback GNE for this class of difference games.

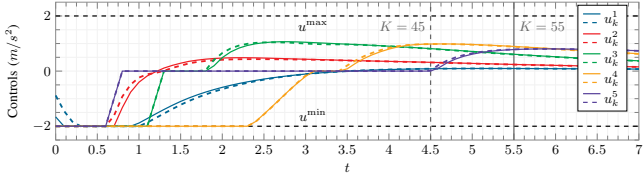
\begin{figure}[t] 
	\centering 
	\input{VP_C2}
    \vspace{-0.1in}
	\caption{Comparison of OL-GNE for $K=45$ and $K=55$. }  
	\label{fig:Fig2}
\end{figure}

\bibliographystyle{IEEEtran}
\bibliography{main}

\end{document}

%% file: VP_2.tex
\pgfplotsset{axis line style={black!45}}
\subfloat[]{\label{fig:VPfig1}
	\scalebox{0.575}{\begin{tikzpicture}[scale=1, >=latex']
	    \begin{axis}[width=15cm, height=4.75cm,  ylabel=Position ($m$),
    yticklabel style={
    text width=1.5em,
    align=right
},  grid=both,
			grid style={line width=.1pt, draw=gray!10}, minor tick num=3, xmin=0, xmax=20,ymin=-20,ymax=250]
			\addplot[MidnightBlue, very thick] table [x index=1, y index=7]{Data/vehicle_platoon3.dat}; 
			\addlegendentry{$p_k^{1}$};
            \addplot[name path=p2top, Red, very thick]table[x index=1, y index=8]{Data/vehicle_platoon3.dat}; 
            \addlegendentry{$p_k^{2}$};
            \addplot[name path=p2bottom,draw=none, forget plot]table[x index=1, y expr=\thisrow{p2}-10]{Data/vehicle_platoon3.dat}; 
            \addplot[Red, fill=Red, fill opacity=0.25, draw=none, forget plot]fill between[of=p2top and p2bottom];
            \addplot[name path=p3top, Green, very thick]table[x index=1, y index=9]{Data/vehicle_platoon3.dat}; 
            \addlegendentry{$p_k^{3}$};
            \addplot[name path=p3bottom,draw=none, forget plot]table[x index=1, y expr=\thisrow{p3}-10]{Data/vehicle_platoon3.dat}; 
            \addplot[Green, fill=Green, fill opacity=0.25, draw=none, forget plot]fill between[of=p3top and p3bottom];
            \addplot[name path=p4top, YellowOrange, very thick]table[x index=1, y index=10]{Data/vehicle_platoon3.dat}; 
            \addlegendentry{$p_k^{4}$};
            \addplot[name path=p4bottom,draw=none, forget plot]table[x index=1, y expr=\thisrow{p4}-10]{Data/vehicle_platoon3.dat}; 
            \addplot[YellowOrange, fill=YellowOrange, fill opacity=0.25, draw=none, forget plot]fill between[of=p4top and p4bottom];
            \addplot[name path=p5top, Violet, very thick]table[x index=1, y index=11]{Data/vehicle_platoon3.dat}; 
            \addlegendentry{$p_k^{5}$};
            \addplot[name path=p5bottom,draw=none, forget plot]table[x index=1, y expr=\thisrow{p5}-10]{Data/vehicle_platoon3.dat}; 
            \addplot[Violet, fill=Violet, fill opacity=0.25, draw=none, forget plot]fill between[of=p5top and p5bottom];
            \draw[-,black!60, thick](5.5,-20) to (5.5,240) node[yshift=-2.5cm,right,black]{$K=55$};
            \draw[-,Red,very thick, dotted](0,54.98) to (20,54.98) node[xshift=-8.5cm,above,black]{};
            \draw[-,Green, very thick, dotted](0,109.95) to (20,109.95) node[xshift=-8.5cm,above,black]{};
            \draw[-,YellowOrange, very thick, dotted](0,164.88) to (20,164.88) node[xshift=-8.5cm,above,black]{};
            \draw[-,Violet, very thick, dotted](0,219.73) to (20,219.73) node[xshift=-8.5cm,above,black]{};
		\end{axis}
\end{tikzpicture}}}  
\\ \vspace{-0.12in}
\subfloat[]{\label{fig:VPfig2}
	\begin{tikzpicture}[scale=0.575, >=latex']
		\begin{axis}[ width=15cm, height=4.75cm,  xlabel=$t$, ylabel=Velocity ($m/s$),
    yticklabel style={
    text width=1.5em,
    align=right
}, grid=both,
			grid style={line width=.1pt, draw=gray!10}, minor tick num=3, xmin=0, xmax=20,ymin=24,ymax=36]
            \draw[-,black!60, very thick, dotted](0,30) to (20,30) node[xshift=-8.5cm,above,black]{$v^{ref}$}; 
            \draw[-,black!100, very thick, dashed](0,35) to (20,35) node[xshift=-3.5cm,below,black]{${v}^{max}$};
            \draw[-,black!100, very thick, dashed](0,25) to (20,25) node[xshift=-3.5cm,above,black]{${v}^{min}$}; 
            \draw[-,black!60, thick](5.5,24) to (5.5,36) node[yshift=-0.5cm,right,black]{$K=55$}; 
			\addplot[MidnightBlue, very thick] table [x index=1, y index=2]{Data/vehicle_platoon3.dat}; 
			\addlegendentry{$v_k^{1}$};
			\addplot[Red,   very thick] table [x index=1, y index=3]{Data/vehicle_platoon3.dat}; 
			\addlegendentry{$v_k^{2}$};
		 \addplot[Green, very thick ] table [x index=1, y index=4]{Data/vehicle_platoon3.dat}; 
			\addlegendentry{$v_k^{3}$};
		\addplot[YellowOrange, very thick ] table [x index=1, y index=5]{Data/vehicle_platoon3.dat}; 
			\addlegendentry{$v_k^{4}$};	
            \addplot[Violet, very thick ] table [x index=1, y index=6]{Data/vehicle_platoon3.dat}; 
			\addlegendentry{$v_k^{5}$};	
		\end{axis}
\end{tikzpicture}} 

%% file: VP_C2.tex
\pgfplotsset{axis line style={black!45}}

\begin{tikzpicture}[scale=0.575,>=latex']
\begin{axis}[
    width=15cm,
    height=4.75cm,
    xlabel=$t$,
    ylabel={Controls ($m/s^2$)},
    yticklabel style={
        text width=1.5em,
        align=right},
    grid=both,
    grid style={line width=.1pt, draw=gray!10},
    minor tick num=3,
    xmin=0,xmax=7,
    ymin=-2.5,ymax=2.5]

\draw[-,black,thick,dashed] (0,2) -- (7,2)
node[xshift=-7.5cm,below] {$u^{\max}$};

\draw[-,black,thick,dashed] (0,-2) -- (7,-2)
node[xshift=-7.5cm,above] {$u^{\min}$};

\draw[-,black!60,thick,dashed] (4.5,-2.5) -- (4.5,2.5)
node[yshift=-0.55cm,left] {$K=45$};

\draw[-,black!60,thick] (5.5,-2.5) -- (5.5,2.5)
node[yshift=-0.55cm,right] {$K=55$};

\addplot[MidnightBlue,thick,forget plot]
table[x index=1,y index=12]{Data/vehicle_platoon3.dat};

\addplot[MidnightBlue,dashed,very thick,forget plot]
table[x index=1,y index=12]{Data/vehicle_platoon4.dat};

\addlegendimage{
legend image code/.code={
    \draw[MidnightBlue,thick] (0,0.10cm)--(0.4cm,0.10cm);
    \draw[MidnightBlue,dashed,thick] (0,-0.10cm)--(0.4cm,-0.10cm);
}}
\addlegendentry{$u_k^{1}$}

\addplot[Red,thick,forget plot]
table[x index=1,y index=13]{Data/vehicle_platoon3.dat};

\addplot[Red,dashed,very thick,forget plot]
table[x index=1,y index=13]{Data/vehicle_platoon4.dat};

\addlegendimage{
legend image code/.code={
    \draw[Red,thick] (0,0.10cm)--(0.4cm,0.10cm);
    \draw[Red,dashed,thick] (0,-0.10cm)--(0.4cm,-0.10cm);
}}
\addlegendentry{$u_k^{2}$}

\addplot[Green,thick,forget plot]
table[x index=1,y index=14]{Data/vehicle_platoon3.dat};

\addplot[Green,dashed,very thick,forget plot]
table[x index=1,y index=14]{Data/vehicle_platoon4.dat};

\addlegendimage{
legend image code/.code={
    \draw[Green,thick] (0,0.10cm)--(0.4cm,0.10cm);
    \draw[Green,dashed,thick] (0,-0.10cm)--(0.4cm,-0.10cm);
}}
\addlegendentry{$u_k^{3}$}

\addplot[YellowOrange,thick,forget plot]
table[x index=1,y index=15]{Data/vehicle_platoon3.dat};

\addplot[YellowOrange,dashed,very thick,forget plot]
table[x index=1,y index=15]{Data/vehicle_platoon4.dat};

\addlegendimage{
legend image code/.code={
    \draw[YellowOrange,thick] (0,0.10cm)--(0.4cm,0.10cm);
    \draw[YellowOrange,dashed,thick] (0,-0.10cm)--(0.4cm,-0.10cm);
}}
\addlegendentry{$u_k^{4}$}

\addplot[Violet,thick,forget plot]
table[x index=1,y index=16]{Data/vehicle_platoon3.dat};

\addplot[Violet,dashed,very thick,forget plot]
table[x index=1,y index=16]{Data/vehicle_platoon4.dat};

\addlegendimage{
legend image code/.code={
    \draw[Violet,thick] (0,0.10cm)--(0.4cm,0.10cm);
    \draw[Violet,dashed,thick] (0,-0.10cm)--(0.4cm,-0.10cm);
}}
\addlegendentry{$u_k^{5}$}

\end{axis}
\end{tikzpicture}

%% file: main.bib
@book{Basar:1999,
	title={Dynamic Noncooperative Game Theory: Second Edition},
	author={Ba{\c{s}}ar, T. and Olsder, G.J.},
	series={Classics in Applied Mathematics},
	year={1999},
	publisher={Society for Industrial and Applied Mathematics}
}

@book{Engwerda:2005,
  title={{LQ} dynamic optimization and differential games},
  author={Engwerda, Jacob},
  year={2005},
  publisher={John Wiley \& Sons}
}

@article{Arrow:1954,
  title={Existence of an equilibrium for a competitive economy},
  author={Arrow, Kenneth J and Debreu, Gerard},
  journal={Econometrica: Journal of the Econometric Society},
  pages={265--290},
  year={1954},
  publisher={JSTOR}
}

@article{Facchinei:2010_a,
  title={Generalized {N}ash equilibrium problems},
  author={Facchinei, Francisco and Kanzow, Christian},
  journal={Annals of Operations Research},
  volume={175},
  number={1},
  pages={177--211},
  year={2010},
  publisher={Springer}
}

@article{Harker:1991,
  title={Generalized {N}ash games and quasi-variational inequalities},
  author={Harker, Patrick T},
  journal={European journal of Operational research},
  volume={54},
  number={1},
  pages={81--94},
  year={1991},
  publisher={Elsevier}
}

@article{Rosen:1965,
  title={Existence and uniqueness of equilibrium points for concave {N}-person games},
  author={Rosen, J Ben},
  journal={Econometrica},
  pages={520--534},
  year={1965}
}

@article{Etesami:2019,
  title={Dynamic games in cyber-physical security: An overview},
  author={Etesami, S Rasoul and Ba{\c{s}}ar, Tamer},
  journal={Dynamic Games and Applications},
  volume={9},
  number={4},
  pages={884--913},
  year={2019},
  publisher={Springer}
}

@article{Jorgensen:2010,
  title={Dynamic games in the economics and management of pollution},
  author={J{\o}rgensen, Steffen and Mart{\'\i}n-Herr{\'a}n, Guiomar and Zaccour, Georges},
  journal={Environmental Modeling \& Assessment},
  volume={15},
  number={6},
  pages={433--467},
  year={2010},
  publisher={Springer}
}

@article{Li:2019,
  title={Differential game theory for versatile physical human--robot interaction},
  author={Li, Yanan and Carboni, Gerolamo and Gonzalez, Franck and Campolo, Domenico and Burdet, Etienne},
  journal={Nature Machine Intelligence},
  volume={1},
  number={1},
  pages={36--43},
  year={2019},
  publisher={Nature Publishing Group UK London}
}

@article{Chen:2014,
  title={Cooperative control of power system load and frequency by using differential games},
  author={Chen, Haoyong and Ye, Rong and Wang, Xiaodong and Lu, Runge},
  journal={IEEE Transactions on Control Systems Technology},
  volume={23},
  number={3},
  pages={882--897},
  year={2014},
  publisher={IEEE}
}

@article{Monti:2024,
  title={Feedback and open-loop {N}ash equilibria for {LQ} infinite-horizon discrete-time dynamic games},
  author={Monti, Andrea and Nortmann, Benita and Mylvaganam, Thulasi and Sassano, Mario},
  journal={SIAM Journal on Control and Optimization},
  volume={62},
  number={3},
  pages={1417--1436},
  year={2024},
  publisher={SIAM}
}

@article{Pappas:1980,
  title={On the numerical solution of the discrete-time algebraic {R}iccati equation},
  author={Pappas, Thrasyvoulos and Laub, Alan and Sandell, Nils},
  journal={IEEE Transactions on Automatic Control},
  volume={25},
  number={4},
  pages={631--641},
  year={1980},
  publisher={IEEE}
}

@article{Partha:2023,
  title={Linear-quadratic mean-field-type difference games with coupled affine inequality constraints},
  author={Mohapatra, Partha Sarathi and Reddy, Puduru Viswanadha},
  journal={IEEE Control Systems Letters},
  volume={7},
  pages={1987--1992},
  year={2023},
  publisher={IEEE}
}

@article{Partha:2026a,
  title={Generalized open-loop {N}ash equilibria in linear-quadratic difference games with coupled-affine inequality constraints},
  author={Mohapatra, Partha Sarathi and Reddy, Puduru Viswanadha},
  journal={IEEE Transactions on Automatic Control},
  volume={71},
  number={3},
  pages={1969--1976},
  year={2026},
  publisher={IEEE}
}

@article{Partha:2026b,
  title={Feedback {S}tackelberg-{N}ash equilibria in difference games with quasi-hierarchical interactions and inequality constraints},
  author={Mohapatra, Partha Sarathi and Reddy, Puduru Viswanadha and Zaccour, Georges},
  journal={IEEE Transactions on Automatic Control},
  volume={71},
  number={6},
  pages={3848--3863},
  year={2026},
  publisher={IEEE}
}

@article{Benenati:2026,
  title={Linear-quadratic dynamic games as receding-horizon variational inequalities},
  author={Benenati, Emilio and Grammatico, Sergio},
  journal={IEEE Transactions on Automatic Control},
  volume={71},
  number={4},
  pages={2404--2417},
  year={2026},
  publisher={IEEE}
}

@article{Baghbadorani:2026,
  title={Fast Newton methods for linear-quadratic dynamic games with application to autonomous vehicle platooning and intersection crossing},
  author={Baghbadorani, Reza Rahimi and Grammatico, Sergio},
  journal={arXiv preprint:2605.01898},
  year={2026}
}

@book{Pang:2009,
  title={The linear complementarity problem},
  author={Cottle, Richard W and Pang, Jong-Shi and Stone, Richard E},
  series={Classics in Applied Mathematics},
  year={2009},
  publisher={SIAM}
}

@article{Engwerda:2000,
  title={Feedback {N}ash equilibria in the scalar infinite horizon {LQ}-game},
  author={Engwerda, Jacob},
  journal={Automatica},
  volume={36},
  number={1},
  pages={135--139},
  year={2000},
  publisher={Elsevier}
}

@article{Engwerda:2007,
  title={Algorithms for computing {N}ash equilibria in deterministic {LQ} games},
  author={Engwerda, Jacob},
  journal={Computational Management Science},
  volume={4},
  number={2},
  pages={113--140},
  year={2007},
  publisher={Springer}
}

@article{Engwerda:2012,
  title={Feedback {N}ash equilibria for linear quadratic descriptor differential games},
  author={Engwerda, Jacob C and others},
  journal={Automatica},
  volume={48},
  number={4},
  pages={625--631},
  year={2012},
  publisher={Elsevier}
}

@article{Engwerda:2009,
  title={The open-loop linear quadratic differential game for index one descriptor systems},
  author={Engwerda, Jacob Christiaan and others},
  journal={Automatica},
  volume={45},
  number={2},
  pages={585--592},
  year={2009},
  publisher={Elsevier}
}

@article{Tanwani:2019,
  title={Feedback {N}ash equilibrium for randomly switching differential--algebraic games},
  author={Tanwani, Aneel and Zhu, Quanyan},
  journal={IEEE Transactions on Automatic Control},
  volume={65},
  number={8},
  pages={3286--3301},
  year={2019},
  publisher={IEEE}
}

@article{Laine:2023,
  title={The computation of approximate generalized feedback {N}ash equilibria},
  author={Laine, Forrest and Fridovich-Keil, David and Chiu, Chih-Yuan and Tomlin, Claire},
  journal={SIAM Journal on Optimization},
  volume={33},
  number={1},
  pages={294--318},
  year={2023},
  publisher={SIAM}
}

@article{Zazo:2016,
  title={Dynamic potential games with constraints: Fundamentals and applications in communications},
  author={Zazo, Santiago and Macua, Sergio Valcarcel and S{\'a}nchez-Fern{\'a}ndez, Matilde and Zazo, Javier},
  journal={IEEE Transactions on Signal Processing},
  volume={64},
  number={14},
  pages={3806--3821},
  year={2016},
  publisher={IEEE}
}

@article{Reddy:2015,
  title={Open-loop {N}ash equilibria in a class of linear-quadratic difference games with constraints},
  author={Reddy, Puduru Viswanadha and Zaccour, Georges},
  journal={IEEE Transactions on Automatic Control},
  volume={60},
  number={9},
  pages={2559--2564},
  year={2015},
  publisher={IEEE}
}

@article{Reddy:2017,
  title={Feedback {N}ash equilibria in linear-quadratic difference games with constraints},
  author={Reddy, Puduru Viswanadha and Zaccour, Georges},
  journal={IEEE Transactions on Automatic Control},
  volume={62},
  number={2},
  pages={590--604},
  year={2017},
  publisher={IEEE}
}

@book{Blot:2014,
  author    = {Blot, Jo{\"e}l and Hayek, Na{\"\i}la},
  title     = {Infinite-horizon optimal control in the discrete-time framework},
  publisher = {Springer Briefs in Optimization. Springer},
  year      = {2014}
}

@article{Aseev:2017,
  title={Optimality conditions for discrete-time optimal control on infinite horizon},
  author={Aseev, SM and Krastanov, MI and Veliov, VM},
  journal={Pure and Applied Functional Analysis},
  volume={2},
  number={3},
  pages={395--409},
  year={2017}
}
